\documentclass[a4paper,11pt,reqno]{amsart}

\usepackage{aeguill} 
\usepackage{enumerate}
\usepackage{amssymb,amsmath,latexsym,amsthm}
\usepackage[a4paper, tmargin=1.5in, bmargin=1.5in, lmargin=1in,rmargin=1in]{geometry}
\usepackage{url}
\usepackage[french, main=english]{babel}
\usepackage[utf8]{inputenc}
\usepackage{mathrsfs}
\usepackage{xcolor}
\usepackage{comment}
\definecolor{violet}{rgb}{0.0,0.2,0.7}
\definecolor{rouge2}{rgb}{0.8,0.0,0.2}
\usepackage{tikz}
\usepackage{empheq}
\usepackage{tikz-cd}
\usetikzlibrary{matrix,arrows,decorations.pathmorphing}
\usepackage{hyperref}
\usepackage{mathpazo} 
\hypersetup{
	bookmarks=true,         
	unicode=false,          
	pdftoolbar=true,        
	pdfmenubar=true,        
	pdffitwindow=false,     
	pdfstartview={FitH},    
	pdftitle={},    
	pdfauthor={},     
	colorlinks=true,       
	linkcolor=rouge2,          
	citecolor=violet,        
	filecolor=black,      
	urlcolor=cyan}           
\usepackage{enumitem}
\usepackage{appendix}

\usepackage{amssymb}
\usepackage{amsmath}
\usepackage{mathrsfs}
\usepackage{amsfonts}
\usepackage{mathtools}
\usepackage{dsfont}

\DeclareSymbolFont{script}{U}{eus}{m}{n}
\DeclareSymbolFontAlphabet{\amathscr}{script}
\DeclareMathSymbol{\Wedge}{0}{script}{"5E}
\DeclareMathAlphabet{\mathrmsl}{OT1}{cmr}{m}{sl}
\theoremstyle{plain}
\newtheorem{thm}{Theorem}[section]

\newtheorem{lem}[thm]{Lemma}
\newtheorem{prop}[thm]{Proposition}

\newtheorem{cor}[thm]{Corollary}

\theoremstyle{plain}
\newtheorem{bigthm}{Theorem}

\newtheorem{taggedbigprobx}{Problem}
    \newenvironment{taggedbigprob}[1]
    {\renewcommand\thetaggedbigprobx{#1}\taggedbigprobx}
    {\endtaggedbigprobx}
    
	\newtheorem{bigconj}[bigthm]{Conjecture}

    \renewcommand{\thebigthm}{\Alph{bigthm}} 
    \newtheorem*{bigrmk*}{Remark}

    \makeatletter
    \newcommand{\settheoremtag}[1]{
    \let\oldthebigthm\thebigthm
    \renewcommand{\thebigthm}{#1}
    \g@addto@macro\endbigconj{
    \addtocounter{bigthm}{-1}
    \global\let\thebigthm\oldthebigthm}
    }
    \makeatother

\theoremstyle{plain}
    
\theoremstyle{definition}
	\newtheorem{defn}[thm]{Definition}
	\newtheorem{eg}[thm]{Example}

	\newtheorem*{claim*}{Claim}
    
	\newtheorem*{ack*}{Acknowledgements}
\theoremstyle{remark}
	\newtheorem{rmk}[thm]{Remark}
	\newtheorem*{rmk*}{Remark}

	\newtheorem*{ans*}{Answer}

\numberwithin{equation}{section}

\newlist{steps}{enumerate}{1}
\setlist[steps, 1]{label = Step \arabic*:}

\mathcode`l="8000
\begingroup
\makeatletter
\lccode`\~=`\l
\DeclareMathSymbol{\lsb@l}{\mathalpha}{letters}{`l}
\lowercase{\gdef~{\ifnum\the\mathgroup=\m@ne \ell \else \lsb@l \fi}}%
\endgroup

\DeclareFontFamily{U}{MnSymbolC}{}
\DeclareSymbolFont{MnSyC}{U}{MnSymbolC}{m}{n}
\DeclareFontShape{U}{MnSymbolC}{m}{n}{
	<-6>  MnSymbolC5
	<6-7>  MnSymbolC6
	<7-8>  MnSymbolC7
	<8-9>  MnSymbolC8
	<9-10> MnSymbolC9
	<10-12> MnSymbolC10
	<12->   MnSymbolC12}{}
\DeclareMathSymbol{\intprod}{\mathbin}{MnSyC}{'270}

\DeclareMathOperator{\Hess}{Hess}

\DeclareMathOperator{\Vol}{Vol}

\DeclareMathOperator{\ext}{ext}

\DeclareMathOperator{\Aut}{Aut}

\DeclareMathOperator{\Scal}{Scal}

\DeclareMathOperator{\Fut}{\mathbf{Fut}}
\DeclareMathOperator{\EH}{\mathbf{EH}}
\DeclareMathOperator{\BFS}{\mathbf{S}}

\DeclareMathOperator{\BFV}{\mathbf{V}}

\def\1{\mathds{1}}

\newcommand{\Sa}{S}
\newcommand{\R}{\mathbb{R}}

\newcommand{\kt}{\hat{\mathfrak{t}}}

\newcommand{\cR}{\mathcal{R}}

\newcommand{\loc}{\mathrm{loc}}

\newcommand{\Pol}{{\rm P}}

\newcommand{\Ds}{{\amathscr D}}

\newcommand\dt{\delta}

\newcommand\vep{\varepsilon}

\newcommand\om{\omega}

\newcommand\af{\alpha}

\newcommand\ld{\lambda}

\newcommand\Dt{\Delta}

\newcommand\Sm{\Sigma}

\newcommand{\cO}{{\mathcal O}}

\newcommand\BN{\mathbb{N}}
\newcommand\BZ{\mathbb{Z}}

\newcommand\BR{\mathbb{R}}

\newcommand\BS{\mathbb{S}}

\newcommand\BP{\mathbb{P}}

\newcommand\CC{\mathcal{C}}

\newcommand\CO{\mathcal{O}}

\newcommand{\Sph}{\mathbb{S}}
\newcommand{\tor}{\mathfrak{t}}

\newcommand{\PP}{{\mathbb P}}

\newcommand\pl{\partial}

\newcommand\dd{\mathrm{d}}

\newcommand{\C}{\mathbb{C}}
\newcommand{\T}{\mathbb{T}}
\newcommand{\Z}{\mathbb{Z}}

\newcommand{\RN}[1]{\textup{\uppercase\expandafter{\romannumeral#1}}}

\title{From Calabi's extremal metrics to scalar-flat K\"ahler cones}

\author{Vestislav Apostolov}
\author{Abdellah Lahdili}
\author{Chung-Ming Pan}

\address[Vestislav Apostolov]{Centre interuniversitaire de recherches en g\'eom\'etrie et topologie (CIRGET), Universit\'e du Qu\'ebec \`a Montr\'eal, and, Institute of Mathematics and Informatics, Bulgarian Academy of Sciences} 
\email{\href{mailto:apostolov.vestislav@uqam.ca}{apostolov.vestislav@uqam.ca}}
\urladdr{\href{http://profmath.uqam.ca/~apostolo/}{http://profmath.uqam.ca/~apostolo/}}

 \address[Abdellah Lahdili]{Campus Saint-Jean, University of Alberta; 8406 91 Street, Edmonton, Alberta, Canada T6C 4G9}
 \email{\href{mailto:lahdili@ualberta.ca}{lahdili@ualberta.ca}}

\address[Chung-Ming Pan]{Centre interuniversitaire de recherches en g\'eom\'etrie et topologie (CIRGET), Universit\'e du Qu\'ebec \`a Montr\'eal; Case postale 8888, Succursale centre-ville, Montr\'eal, Qu\'ebec, H3C 3P8, Canada}
\email{\href{mailto:pan.chung_ming@uqam.ca}{pan.chung\_ming@uqam.ca}; \href{mailto:bandan770@gmail.com}{bandan770@gmail.com}}
\urladdr{\href{https://chungmingpan.github.io/}{https://chungmingpan.github.io/}}


\begin{document} 

\maketitle

\begin{abstract} 
We prove that for any smooth polarized complex $n$-dimensional manifold $(X, L_X)$ which admits an extremal K\"ahler metric in $c_1(L_X)$, and for any integer $k$ large enough (in terms of a bound depending on $(X, L_X)$), the $(n+k+1)$-dimensional complex cone $\mathcal{Y}:= \overline{(L_X \otimes \mathcal{O}_{\mathbb{P}^k}(1))^{\times}}$ with section $X \times \mathbb{P}^k$ admits a scalar-flat K\"ahler cone metric. 
Equivalently, the unweighted Sasaki join of a smooth compact quasi-regular extremal Sasaki manifold with a regular Sasaki sphere $\mathbb{S}^{2k+1}$ of sufficiently large dimension $(2k+1)$ admits a Sasaki metric of constant (positive) scalar curvature. This gives an affirmative answer to an asymptotic version of a question raised by Boyer--Huang--Legendre--T{\o}nnesen-Friedman in \cite{BHLT0}.
\end{abstract}

\tableofcontents

\section*{Introduction} 
Let $\overline{Y} \subset \C^N$ be an affine cone such that $Y:= \overline{Y}\setminus\{0\}$ is smooth, and $\hat\T \subset {\rm GL}_N(\C)$ be a maximal compact torus preserving $\overline{Y}$. 
A $\hat\T$-invariant cone K\"ahler metric $\hat \omega=\frac{1}{4}dd^c r^2>0$ on $Y$, polarized by an element $\hat\xi \in {\rm Lie}(\T)$,  is defined by a positive smooth function $r$  on $Y$, which satisfies $-{\mathcal L}_{J\hat \xi} r = r$. The existence on $Y$ of a Ricci-flat $\hat\T$-invariant cone K\"ahler metric or, more generally, which has zero scalar curvature, is by now a well-rooted problem extensively studied both from the point of view of Sasaki geometry ~\cite{BG, MSY, BvC, CSz1, CSz2} and weighted cscK geometry~\cite{AC, ACL, ChiLi, Apostolov_Lahdili_Legendre_2024}. Scalar-flat K\"ahler cones also fit in the general theory of singular constant scalar curvature K\"ahler metrics that have been studied recently in \cite{Pan_To_Trusiani_2023, Boucksom_Jonsson_Trusiani_2024, Pan_To_2024}, where they are expected to arise as bubbling limits of singular cscK metrics,  and thus play a key role in understanding their behaviour near singularities. 


\smallskip
A prototypical example of an  K\"ahler cone arises in the case when $(X, L)$ is a smooth polarized variety and $\overline{Y}:= \overline{(L^{-1})^{\times}}$ is the cone obtained by contracting the zero section of the line bundle $L^{-1}$ to a point. 
In this case, one can take $\hat\T \subset {\rm Aut}(X, L)$ to be a maximal compact torus inside the group of automorphisms of $(X, L)$. 
If we denote by $\hat{\chi}_0 \in {\rm Lie}(\hat\T)$ the generator of $S^1$-action on the fibres of $L^{-1}$, it is well-known~\cite{BG} that $X$ admits a K\"ahler metric $\omega \in 2\pi c_1(L)$ of constant positive scalar curvature if and only if $Y$ admits a $\hat\T$-invariant cone K\"ahler metric of zero scalar curvature polarized by a positive multiple of $\hat\chi_0$. 
However, there are many known examples (see e.g. \cite{GMSW, FOW, BT0}) of smooth polarized varieties $(X, L)$ which do not admit cscK metrics but nevertheless the associated cone $(Y, \hat\T)$ admits a scalar-flat cone K\"ahler metric polarized by $\hat \xi \in {\rm Lie}(\hat\T)$ different than a multiple of $\hat \chi_0$. 
The study of scalar-flat cone K\"ahler metrics on $(Y, \hat\T)$ thus naturally (and non-trivially) extends the much studied theory of cscK metrics on polarized varieties.

\smallskip
It is well-known~\cite{CSz1,BvC,ACL}  that the subspace $\hat\tor_+(Y)\subset {\rm Lie}(\hat\T)$ of elements $\hat \xi \in {\rm Lie}(\hat\T)$ such that $Y$ admits a $\hat\xi$-polarized $\hat\T$-invariant  cone K\"ahler metrics is an open convex simplicial affine cone in ${\rm Lie}(\hat\T)$, called the \emph{Sasaki--Reeb cone} of $(Y, \hat \T)$. 
For each $\hat \xi \in \hat\tor_+(Y)$, there exists a linear map  ${\Fut}_{\hat\xi} : {\rm Lie}(\hat\T) \to {\mathbb R}$, reminiscent to the Futaki invariant in the K\"ahler context, such that ${\Fut}_{\hat\xi}\equiv 0$ is a necessary condition for the existence of a $\hat\xi$-polarized scalar-flat $\hat\T$-invariant K\"ahler metric on $Y$ (see \cite{BGS, FOW, BHL}). 
A remarkable result from \cite{BHL} states that the subspace ${\mathcal F}:=\{\hat \xi \in \hat\tor_+(Y)\, | \, {\Fut}_{\hat \xi} \equiv 0\}$ is  non-empty. 
This extends a key observation going back to \cite{MSY} in the case of Calabi--Yau cones, where it is also shown that  $\mathcal{F}$ then consists of a unique ray of polarizations. 
Thus, unlike the K\"ahler context, the existence of scalar-flat K\"ahler cone metrics on $Y$ is Futaki unobstructed for suitably chosen polarizations. 

\smallskip
Besides the vanishing of ${\Fut}_{\hat \xi}$, the results in \cite{Lahdili_2019, AC} also yield obstructions for the existence of scalar-flat cone metrics on $Y$  expressed as a Calabi--Lichnerowicz--Matsushima type constraint on the group of complex automorphisms of any orbifold quotient $X= Y/\C^*_{\hat\chi}$ of $Y$ by the $\C^*$-action generated by a \emph{quasi-regular} element $\hat\chi\in \hat\tor_+(Y)$. 
This constraint is satisfied if $X$ admits an extremal orbifold K\"ahler metric in the Chern de Rham class of the induced polarization $L$. 
The following is a ramification of a general  problem proposed in \cite{BHLT0}

\begin{taggedbigprob}{1}[{\cite{BHLT0}}]\label{main-problem} 
Suppose $(\overline{Y}, \hat \T)$ is an affine cone which admits a quasi-regular Sasaki--Reeb vector field $\hat \chi \in \hat\tor_+(Y)$ such that the orbifold $X=Y/\C^*_{\hat \chi}$ admits an extremal K\"ahler metric in $2\pi c_1(L)$. 
Does $(Y, \hat\T)$ admit a $\hat\T$-invariant scalar-flat K\"ahler metric polarized by some $\hat\xi \in \mathcal{F}\subset \hat\tor_+$?
\end{taggedbigprob}

It turns out~\cite{BHLT1} that the more subtle $K$-stability obstruction theory~\cite{CSz1,ACL} yields examples of cones $(Y, \hat\T)$ where the answer of Problem~\ref{main-problem} is ``No''.


\bigskip
Motivated by these considerations, we establish in this paper the following `asymptotic' affirmative version of Problem~\ref{main-problem}.

\begin{bigthm}\label{thm:main} In the setup of Problem~\ref{main-problem}, there exists a positive integer $k_0=k_0(X, L)$, such that for any integer $k \geq k_0$, the K\"ahler cone  $\mathcal{Y} = (\mathcal{L}^{-1})^{\times}$ associated  to $\mathcal{X}= X \times \PP^k$  polarized by the line bundle   $\mathcal{L}= \pi_X^* L \otimes \pi_{\PP^k}^* \cO_{\PP^k}(1)$  admits a scalar-flat K\"ahler cone metric.
\end{bigthm}

It is well-known~\cite{BG} that there is a bijective correspondence between smooth polarized K\"ahler cones and compact $(2n+1)$-dimensional Sasaki manifolds, such that scalar flat  K\"ahler cones correspond to Sasaki manifolds with transversal scalar curvature equal to $2n(n+1)$.  
Thus, Theorem~\ref{thm:main} above can be restated as the following assertion about Sasaki manifolds, relating to the original motivation in \cite{BHLT0}:

\begin{bigthm}\label{thm:Sasaki} Suppose $(\Sa, \hat \chi, \Ds, J)$ is a smooth quasi-regular compact $(2n+1)$-dimensional Sasaki manifold such that the transversal K\"ahler geometry is Calabi-extremal.  
Then, there exists a number  $k_0=k_0(\Sa, \hat \chi, \Ds, J)>0$ such that for any integer $k\geq k_0$, the Sasaki join $\Sa \, \boldsymbol{\star}_{1,1} \, \mathbb{S}^{2k+1}$ of $(\Sa, \hat \chi_0, \Ds_0, J_0)$ with the standard (regular) Sasaki sphere $(\Sph^{2k+1}, \hat \chi_0, \Ds_0, J_0)$ admits a Sasaki--Reeb vector field $\hat \xi_k$ and a compatible Sasaki structure of constant positive transversal scalar curvature.
\end{bigthm}

Similar existence results, formulated for special choices of $\Sa$ and taking weighted Sasaki joins with weighted Sasaki $3$-spheres, appear in \cite{BT1}. 
We emphasize that our statement above holds for \emph{any} transversally extremal Sasaki manifold and that we take \emph{unweighted} Sasaki join with the \emph{regular} Sasaki sphere over $(\PP^k, \cO_{\PP^k}(1))$ of sufficiently large dimension $2k+1 \geq 2k_0 +1$. 
The price for such a general existence result, however, is that our proof does not provide a quantitative estimate of the value of $k_0$. On the other hand,  the arguments we develop here are sufficiently robust (see Theorem~\ref{thm:SFC} for a general statement) to also allow us construct Sasaki metrics of \emph{negative} constant transversal scalar curvature on the Sasaki join of any extremal quasi-regular Sasaki manifold with an $\Sph^1$-bundle over the product of sufficiently many copies of a Riemann surface of sufficiently large genus (see Corollary~\ref{c:negative cscS}).

\smallskip
A related result to Theorem~\ref{thm:main} above was recently obtained in \cite{Apostolov_Lahdili_Legendre_2024},  where the authors proved that if $X$ is a compact Fano orbifold which admits a K\"ahler Ricci soliton, then the canonical cone $\overline{K^{\times}_{X\times \PP^k}}$ of $X\times \PP^k$ admits a Ricci-flat K\"ahler cone metric for $k \geq k_0$.  Notice that K\"ahler Ricci solitons are self-similar solutions to the normalized K\"ahler Ricci flow, whereas extremal K\"ahler metrics are self-similar solutions of the Calabi flow. 

\smallskip
We now briefly outline the strategy of proof of Theorem~\ref{thm:main}.
Similarly to \cite{Apostolov_Lahdili_Legendre_2024}, the main idea is to reformulate the existence problem for a scalar-flat cone K\"ahler metric on $\mathcal{Y}$ as a weighted constant scalar curvature problem on ${\mathcal X} = X \times \PP^k$, using the framework developed in \cite{Lahdili_2019, AC, ACL} (see Corollary~\ref{c:tool} below for a precise statement).
By looking for weighted constant scalar curvature metrics on $\mathcal{X}$ which are products of a K\"ahler metric on $X$ with a Fubini--Study metric of constant scalar curvature $2k(k+1)$ on $\PP^k$, we further reduce the problem to finding $(v_{\lambda_k, b_k, k}, w_{\lambda_k, b_k, k})$-cscK metrics on $(X, 2\pi c_1(L))$. 
Here the sequence of weights $(v_{\lambda_k, b_k, k}, w_{\lambda_k, b_k, k})$ are defined over the momentum polytope $\Pol$ of $(X, L, \hat\T)$ and depend on the dimension $k$ of the factor $\PP^k$ of ${\mathcal X}$, a single affine-linear function $\ell_k$ over $\Pol$, and two real constants $\lambda_k, b_k$ coming from the geometric freedom of the construction. 

The main technical challenge is to ensure the vanishing of the corresponding Futaki invariant for the weight functions $(v_{\lambda_k, b_k, k}, w_{\lambda_k, b_k, k})$, while extracting an a priori asymptotic limit of the weights $(v_{\lambda_k, b_k, k}, w_{\lambda_k, b_k, k})$ as $k\to \infty$ to the extremal weights $(1, \ell_{\rm ext})$. 
To this end, we introduce a family of  (suitably normalized) Einstein--Hilbert type functionals $\widetilde{\bf EH}_{\lambda,b, k}$, defined on the space of Sasaki--Reeb polarizations of $Y$ (a finite dimensional convex affine cone), whose critical points determine the affine-linear function $\ell= \ell_{\lambda, k}$ and the constant $b= b_{\lambda, k}$ in the definition of $(v_{\lambda, b, k}, w_{\lambda, b, k})$ in such a way that the corresponding Futaki obstruction of $(v_{\lambda, b, k}, w_{\lambda, b, k})$ vanishes. 
The leading term of $\widetilde{\EH}_{\lambda,b, k}$ is related to some other functionals introduced and used in \cite{BHL, FO, Inoue_2022, Lahdili_Legendre_Scarpa_2023}. 

Our main observation (see Proposition~\ref{prop:EH_C2loc} and Lemma~\ref{l:conv} below) is that, as $k\to \infty$ and $b \to 0$, the functionals $\widetilde{\bf EH}_{\lambda, b, k}$ converge uniformly in the local $\CC^2$-topology to the entropy functional $W_{\lambda}$ introduced and studied by E. Inoue \cite{Inoue_2022} in the context of $\mu^{\lambda}_{\xi_{\lambda}}$-cscK metrics.  Furthermore, similarly to this latter work, as $(\lambda, k) \to (\infty, \infty)$, $\widetilde{\bf EH}_{\lambda, b_{\ld,k},  k}$ converges uniformly in the local $\CC^2$-norm to a concave functional whose maximizer corresponds to the extremal affine-linear function $\ell_{\rm ext}$ of $(X, L, \hat \T)$.  Appealing to the implicit function theorem, we thus produce a sequence of weight functions $(v_{\lambda_k,b_k, k}, w_{\lambda_k, b_k, k})$ with vanishing weighted Futaki invariant which converges as $k\to \infty$ to the weight functions $(1, \ell_{\rm ext})$ corresponding to the extremal K\"ahler metric in $2\pi c_1(L)$.
Finally, a LeBrun--Simanca type openness result implies that, for all sufficiently large $k$, $X$ admits a $(v_{\lambda_k,b_k, k}, w_{\lambda_k,b_k, k})$-cscK metric, and thus, ${\mathcal Y}$ carries a scalar-flat K\"ahler cone metric.

\begin{ack*}
The authors are grateful to E. Legendre for sharing with us her expertise and for answering our questions on Sasaki geometry. 
The authors thank C. Boyer, C. T{\o}nnesen-Friedman and H. Huang for their comments on the manuscript.

V. Apostolov was supported by an NSERC discovery grant and an FRQNT team grant. 
A. Lahdili was supported in part by a CIRGET postdoctoral fellowship during the preparation of this work. 
C.-M. Pan is supported by a CIRGET postdoctoral fellowship. 
\end{ack*}

\section{K\"ahler cones and Sasaki manifolds} 
In this section, we recall some by now standard material relating compact Sasaki manifolds $\Sa$ of dimension $(2n+1)$ to K\"ahler cones of complex dimension $(n+1)$. 
Basic references are \cite{BG, MSY, FOW, CSz1}. 
We shall adopt the point of view of \cite{AC,ACL}, which will allow us to recast the possibly irregular transversal K\"ahler geometry of a Sasaki manifold in terms of the K\"ahler geometry of a given regular or quasi-regular quotient.

\subsection{Sasaki structures}\label{s:sasaki} 
We consider the following general setup:  $(\Sa, \Ds_0, J_0)$ is a compact $(2n+1)$-dimensional strictly pseudo-convex CR manifold invariant under the action of a compact torus $\hat \T$ whose Lie algebra is denoted by $\hat{tor} $. 
We say that $\hat\xi \in \hat\tor$ is a \emph{Sasaki--Reeb vector field} if $\hat \xi$ is transversal to $\Ds_0$ and the corresponding contact $1$-form  $\eta_0^{\hat \xi}$, which vanishes on $\Ds$ and is equal to $1$ when evaluated at $\hat \xi$,  defines a transversal K\"ahler form $d\eta_0^{\hat \xi}$ on $(\Ds_0, J_0)$, i.e. $(d\eta_0^{\hat \xi})_{|_{\Ds_0}} >0$. 
We denote by $\kt_+(\Sa) \subset \kt$ the \emph{Sasaki--Reeb} cone of Sasaki--Reeb vector fields, and assume that $\kt_+(\Sa)$ is non-empty. 
For any $\hat \xi \in \kt_+(\Sa)$, the data $(\hat\xi, \eta_0^{\hat \xi}, \Ds_0, J_0)$ is referred to as a \emph{Sasaki structure} on $\Sa$.

\begin{eg}[(Quasi-)regular Sasaki structures]\label{ex:regular}
A basic class of examples of Sasaki manifolds is given by the so-called \emph{regular Sasaki} structures, which can be described as follows. 
Let $X$ be a smooth compact complex manifold polarized by a line bundle $L$ and let $\hat \T \subset {\rm Aut}(L, X)$ be a fixed maximal torus covering a maximal torus $\T\subset {\rm Aut}_{r}(X)$, the reduced group of complex automorphisms of $X$. 
Let $\omega_0 \in 2\pi c_1(L)$ be a $\T$-invariant K\"ahler metric on $X$. 
Consider the unitary $\Sph^1$-bundle $\Sa_{\omega_0} \subset L^{-1}$ associated with the hermitian metric $H_{\omega_0}$ on $L^{-1}$ whose curvature is $-i\omega_0$. 
The torus $\hat \T$ naturally acts on $\Sa_{\omega_0}$ preserving the induced CR structures $(\Ds_0, J_0)$. 
If $\hat\chi \in \kt$ denotes the generator of the $\Sph^1$-action along the fibres, then $\hat \chi \in \kt_+(\Sa)$. 
The corresponding $1$-form $\eta_0^{\hat \chi}$ is the unique connection $1$-form on $\Sa_{\omega_0}$ with curvature satisfying 
\[
    d\eta_0^{\hat \chi} = \pi^* \omega_0
\] 
and the distribution $\Ds_0$ coincides with the horizontal distribution of $\eta_0^{\hat \chi}$.  
The transversal K\"ahler structure in this case is simply the pull-back of $\omega_0$ to $(\Ds_0, J_0)$. 
Note that $X= \Sa_{\omega_0}/\Sph^1_{\hat \chi}$, which allows us to recover the K\"ahler structure $\omega_0$ on $X$ from the corresponding Sasaki structure $(\hat \chi, \eta_0^{\hat \chi}, \Ds_0, J_0)$ on $\Sa_{\omega_0}$. 
This correspondence also applies to periodic Sasaki--Reeb vector fields $\hat\xi \in \kt_+(\Sa)$. 
In this case, $(\Sa, \hat \xi, \Ds_0, J_0)$ is the (smooth) total space of an $\Sph^1$-orbifold bundle over the K\"ahler orbifold 
\[
    (\tilde X:=\Sa/\Sph^1_{\hat \xi}, \tilde \omega_0).
\]  
The principal $\Sph^1$-bundle gives rise to a positive holomorphic orbifold bundle $\tilde L$ on $\tilde X$ with $\tilde \omega_0 \in 2\pi c_1(\tilde L)$. 
This situation is referred to in the literature as a \emph{quasi-regular} Sasaki structure.
\end{eg}

We now introduce the natural variation spaces of Sasaki structures on $\Sa$: 
Fix $\hat \xi\in\kt_+(\Sa)$ on $(\Sa, \Ds_0, J_0)$ and let $\eta^{\hat \xi}_0$ denote the associated contact form. 
We consider the following space of functions  
\[
    \Xi_{\hat \xi,\eta_0^{\hat \xi},J^{\hat \xi}}^{\hat \T}(\Sa) := \left\{ \varphi\in \CC^\infty(\Sa)^{\hat \T} \, \, | \, \, \eta_{\varphi}^{\hat \xi}:=\eta_0^{\hat \xi}+d^c_{\hat \xi} \varphi \, \, \mathrm{satisfies} \, \, d\eta_{\varphi}^{\hat \xi} >0 \mbox{ on } \Ds_{\varphi}:=\ker (\eta_{\varphi}^{\hat\xi})\right\}. 
\]
Here $J^{\hat\xi}\in{\rm End}(T\Sa)$ denotes the extension of $J_0\in {\rm End}(\Ds_0)$ such that $J^{\hat \xi}(\hat\xi) = 0$ and 
$d^c_{\hat \xi} \varphi := -d\varphi \circ J^{\hat \xi}$. 
The positivity of $d\eta^{\hat \xi}_{\varphi}$ on $\Ds_{\varphi}$ is understood with respect to $J^{\hat \xi}_{|_{\Ds_{\varphi}}}$. 
We refer to \cite[\S 2]{ACL} for the basic properties of this space. 

\smallskip
In particular, for any $\varphi \in \Xi_{\hat \xi,\eta_0^{\hat \xi},J^{\hat \xi}}^{\hat \T}(\Sa)$, the quadruple $(\hat \xi, \eta_{\varphi}^{\hat \xi}, \Ds_{\varphi}, J^{\hat \xi}_{|_{\Ds_{\varphi}}})$ defines a Sasaki structure. 
The space $\Xi_{\hat \xi,\eta_0^{\hat \xi},J^{\hat \xi}}^{\hat \T}(\Sa)$ is called a \emph{slice of $(\hat \xi, J^{\hat \xi})$-compatible $\hat \T$-invariant Sasaki potentials}. 
Throughout this paper, we shall use the following identification established in \cite{HS, ACL}.

\begin{prop}\label{sasaki-identification} 
For any Sasaki--Reeb vector fields $\hat \chi, \hat \xi \in \kt_+(\Sa)$, there exists a natural bijection
\[
    \Theta_{\hat \chi, \hat \xi} : \Xi_{\hat \xi,\eta_0^{\hat\xi},J^{\hat \xi}}^{\hat \T}(\Sa) \cong \Xi_{\hat \chi,\eta_0^{\hat \chi},J^{\hat \chi}}^{\hat \T}(\Sa). 
\]
\end{prop}

\begin{rmk}\label{r:sasaki-identification}
In the case of a (quasi-)regular Sasaki structure $(\Sa, \hat \chi, \eta_0^{\hat \chi}, \Ds_0, J_0)$ (see Example~\ref{ex:regular}), one has the identification
\[  
    \Xi_{\hat\chi,\eta_0^{\hat\chi},J^{\hat\chi}}^{\hat \T}(\Sa) \equiv \mathcal{H}^{\T}_{\omega_0}(X):= \left\{\varphi \in \CC^{\infty}(X)^{\T} \, \mid \, \omega_{\varphi}
    = \omega_0 + dd^c\varphi >0 \right\}.
\]
Together with Proposition~\ref{sasaki-identification}, this identification allows one to study $(\hat \xi, J^{\hat\xi})$-compatible $\hat \T$-invariant Sasaki structures on $N$ via the induced K\"ahler geometry of $(X, \T, \alpha=2\pi c_1(L))$.
\end{rmk}

\subsection{K\"ahler cones vs Sasaki structures}\label{s:Kahler-cone}
A smooth complex cone \cite{CSz1,ACL} $(Y,J, \hat\xi)$ is a non-compact $(n+1)$-dimensional complex manifold endowed with a free holomorphic  $\R^+$-action (sometimes denoted by $\R^+_{-J\hat \xi}$) generated by the flow of a real holomorphic vector field $-J\hat\xi$.  
We shall further assume that the quotient $\Sa := Y/\R_{-J\hat \xi}^+$ is a compact $(2n+1)$-dimensional manifold. 
A typical example is given by $Y= L^{\times}$, where $L$ is a holomorphic line bundle over $X$ and $L^{\times}$ denotes the total space of $L$ with its zero section removed. 
In this case, one can take $\hat \chi$ to be the generator of the natural fiberwise $\Sph^1$-action on $L$, so that $L^{\times}/\R_{-J\hat \chi}^{+}= \Sa$ is the unit circle bundle, as described in Example~\ref{ex:regular}. 

\smallskip
We shall further fix a compact torus $\hat\T \subset {\rm Aut}(Y)$ and consider free $\R^+_{-J\hat \xi}$-actions for $\hat \xi \in \kt$, where $\kt = {\rm Lie}(\hat \T)$. 
For any such $\hat\xi \in \kt$, we consider the space of functions
\[
    \mathcal{R}_{\hat \xi}^{\hat \T}(Y):= \left\{r\in \CC^\infty(Y, \R_{>0})^{\hat \T} \,|\, \mathcal{L}_{-J\hat\xi} r =r, \, \, \hat \omega_r := {\frac{1}{4}} dd^c r^2 >0 \right\}.
\]
We assume that $\cR_{\hat \xi}^{\hat \T}(Y)\neq \emptyset$ for at least one $\hat \xi \in \kt$. 
For such a $\hat \xi$, any $r\in \cR_{\hat \xi}^{\hat \T}(Y)$ gives rise to a K\"ahler cone metric $\hat \omega_r = \frac{1}{4}dd^cr^2$, i.e. satisfying $\mathcal{L}_{-J\hat \xi} \hat\omega_r = 2\hat\omega_r$ with respect to the $\R^+$-action generated by $-J\hat \xi$.  
We refer to the pair $(Y, \hat \xi)$ as a \emph{polarized K\"ahler cone} and to $\hat \xi$ as a \emph{polarization}. 
We denote by $\kt_+(Y) \subset \kt$ the affine cone of polarizations of $Y$. 

\begin{eg}[(Quasi-)regular K\"ahler cones]
An important example is given by \emph{regular K\"ahler cones}.  
In this case, one has $Y= (L^{-1})^{\times}$, where $L$ is a polarization of $X$, and $\hat \chi$ denotes the generator of the fibrewise $\Sph^1$-action on $L^{-1}$.  
Given a K\"ahler form $\om \in 2\pi c_1(L)$, let $H_\om$ be the hermitian metric on $L^{-1}$ with curvature $-i \om$. 
The associated fibre-norm $r_{\omega} = \| \cdot\|_{H_{\omega}}$ then defines a K\"ahler cone structure on $Y$ polarized by $\hat{\chi}$, cf. Example~\ref{ex:regular}. 
Furthermore, if we fix a maximal torus $\hat \T \subset {\rm Aut}_0(X,L)$, which projects to a maximal torus $\T \subset {\rm Aut}_{\rm r}(X)$ in the reduced automorphisms group of $X$, then $r_{\omega} \in {\mathcal R}^{\hat\T}_{\hat \chi}(Y)$ provided that $\omega \in 2\pi c_1(L)(X)$ is $\T$-invariant. 
The same construction extends to the quasi-regular case.
\end{eg}

Any $r_0\in \cR_{\hat \xi}^{\hat \T}(Y)$ induces a Sasaki structure on $\Sa=Y/\R^{+}_{-J\hat \xi}$ as follows: 
The torus action of $\hat \T$ descends to $\Sa$, and $\hat \xi \in \kt$ defines a vector field on $\Sa$. 
The $1$-form $\eta_0^{\hat \xi} := d^c \log r_0$ on $Y$ is $-J\hat\xi$-basic, so one can view it  as a $1$-form on $\Sa$. 
Furthermore, the distribution $\Ds_0:= {\rm ker}(\eta_0^{\hat \xi}) \subset T\Sa$ inherits a CR structure $J_0$, defined through the almost complex structure on $\langle \hat \xi, J\hat \xi \rangle^{\perp_{\omega_{r_0}}}\subset TY$. 
Thus, the data $(\hat \xi, \eta_0^{\hat \xi}, \Ds_0, J_0, \hat \T)$ is a Sasaki structure on $\Sa$, as defined in the previous subsection. 
Furthermore, for any other $r \in \cR_{\hat \xi}^{\hat \T}(Y)$, the smooth function $\varphi := \log(r/r_0)$ is $(-J\hat \xi)$-invariant and defines an element of $\Xi_{\hat \xi,\eta_0^{\hat \xi},J^{\hat \xi}}^{\hat \T}(\Sa)$. 
It follows that   
\[
    \kt_+(Y)=\kt_+(\Sa)=:\kt_+, \qquad  \cR_{\hat \xi}^{\hat \T}(Y) \cong \Xi_{\hat \xi,\eta_0^{\hat \xi},J^{\hat \xi}}^{\hat \T}(\Sa), \qquad \forall \hat \xi \in \kt_+.
\]
These identifications are discussed at length, for example, in \cite[\S 1, \S 2]{ACL}. 
In the quasi-regular case, when $(Y, J, \hat \chi)= (L^{-1})^{\times}$ for a polarization $L$ of $X$, the above discussion together with Proposition~\ref{sasaki-identification} provides the following sequence of identifications:
\begin{equation} \label{cone-identification}
    \cR_{\hat \xi}^{\hat \T}(Y) 
    \cong \Xi_{\hat \xi,\eta_0^{\hat \xi},J^{\hat \xi}}^{\hat \T}(\Sa) 
    \cong \Xi_{\hat \chi,\eta_0^{\hat \chi},J^{\hat \chi}}^{\hat \T}(\Sa) 
    \equiv \mathcal{H}^{\T}_{\omega_0}(X), 
    \qquad \forall \hat \xi \in \kt_+. 
\end{equation}

\subsection{Sasaki join}\label{subsec:Sasaki_joint} 
We recall a special case of the definition from \cite{BGO} of the \emph{Sasaki join} construction.  

\begin{defn}[The Sasaki join construction] 
Let $(\Sa_1, \hat \xi_1, \Ds_1, J_1)$ be a \emph{quasi-regular} compact smooth Sasaki manifold of dimension $(2n_1 +1)$, and 
$(\Sa_2, \hat{\xi_2}, \Ds_2, J_2)$ be a \emph{regular} compact smooth Sasaki manifold of dimension $(2n_2+1)$. 
Denote by $(X_1, \omega_1, L_1)$ the K\"ahler orbifold structure obtained on $X_1:=\Sa_1/\Sph^1_{\hat\xi_1}$, and by $(X_2, \omega_2, L_2)$ the smooth K\"ahler structure on $X_2:= \Sa_2/\Sph^1_{\hat \xi_2}$. 
The Sasaki join $\Sa_1 \boldsymbol{\star}_{1,k} \Sa_2$ (of weight $(1, k), \, k \in \Z_{>0}$) is the quasi-regular Sasaki manifold of dimension $(2(n_1+n_2)+1)$, corresponding to the total space of the principle $\Sph^1$ orbifold bundle over $X_1\times X_2$, endowed with Sasaki--Reeb vector field $\hat \xi=\frac{1}{2}(\hat \xi_1 + \frac{1}{k}\hat \xi_2)$ and connection $1$-form whose curvature is $\omega_1\oplus k\omega_2$, see Example~\ref{ex:regular}. 
\end{defn}

It is not immediately clear from the above that $\Sa$ is a smooth Sasaki manifold, but one can also identify $\Sa$ (see \cite{BGO}) with 
\[
    \Sa = (\Sa_1 \times \Sa_2)/\Sph^1_{\hat \zeta_0}, \qquad \hat \zeta_0:=-k\hat \xi_1 + \hat \xi_2,
\]
where the diagonal $\Sph^1_{\hat \zeta_0}$-action is free on the second factor, and thus $\Sa$ is smooth (see \cite[Proposition~2.3]{BGO}.
Note that $\Sa= \Sa_1\boldsymbol{\star}_{1,k} \Sa_2$ is the Sasaki manifold associated to the cone
\[ 
    Y = \left((L_1 \otimes L_2^k)^{-1}\right)^{\times},
\]
where $L_1\otimes L_2^k$ is seen as a polarization of the orbifold $X=X_1\times X_2$.
Thus, the join construction in Sasaki geometry is the analog of the product of polarized K\"ahler manifolds.

\subsection{Sasaki structures with constant transversal scalar curvature versus scalar-flat K\"ahler cones}
We recall here the well-known relationship between the scalar curvature $\Scal_{\hat \omega}$ of a cone K\"ahler metric $\hat \omega = \frac{1}{4} dd^c r^2, \, r\in \mathcal{R}^{\hat \T}_{\hat \xi}(Y)$ and the scalar curvature $\Scal_{\hat \xi}$ of the transversal K\"ahler metric on the  corresponding Sasaki manifold $(\Sa= Y/\R^+_{-J\hat \xi}, \eta_r = d^c \log r, \hat \xi)$:
\[ \Scal_{\hat \omega} = 2 r^2\left(\Scal_{\hat \xi} - 2n(n+1)\right).\]
This can be derived, for instance, from  \cite[Theorem~7.3.12, Lemma~11.1.5]{BG},  noting that there is a difference of factor $1/2$ in the definition of transversal K\"ahler metric used in that reference compared to our conventions above.

\smallskip
Notice that if $(\hat\xi, \Ds, J)$ is a Sasaki structure on $\Sa$ with transversal Scalar curvature $\Scal_{\hat \xi}$, then the Sasaki structure $(\lambda \hat \xi, \Ds, J)$ has transversal scalar curvature $\Scal_{\lambda \hat \xi} = \lambda \Scal_{\hat \xi}$. We thus have from the above discussion
\begin{lem}\label{l:cscS-cones} 
There exists a bijective correspondence between rays of Sasaki--Reeb vectors $[\hat \xi] \in \hat\tor_+(\Sa)$ such that $\Sa$ admits a Sasaki structure with  constant positive transversal scalar curvature $(\hat \xi, \Ds_{\varphi}, J^{\hat \xi})$ for $\varphi \in \Xi_{\hat \xi,\eta_0^{\hat\xi},J^{\hat \xi}}^{\hat \T}(\Sa)$, and polarizations $\hat \xi \in \hat\tor_+(Y)$ such $Y$ admits a $\hat\xi$-polarized cone K\"ahler metric  with zero scalar curvature $\omega=\frac{1}{4} dd^c r$  for some $r\in \mathcal{R}_{\hat \xi}^{\hat\T}(Y)$.
\end{lem}

\section{Weighted cscK metrics} 
\subsection{$(v,w)$-cscK metrics}\label{subsec:vwcscK} 
In this section, $X$ will denote a compact K\"ahler manifold or orbifold, endowed with a K\"ahler class $\alpha \in H^{1,1}(X,\mathbb{R})$ and $\T\subset {\rm Aut}_{\rm r}(X)$ will be a fixed maximal torus of the reduced group of automorphisms ${\rm Aut}_{\rm r}(X)$ (see e.g. \cite{gauduchon-book})  with associated fixed momentum polytope $\Pol_{\alpha} \subset \tor^*:= {\rm Lie}(\T)^*$~\cite{Atiyah, GS, LT}.  
The theory is well-established in the case when $X$ is smooth~\cite{gauduchon-book} and readily extends to the K\"ahler orbifold case using Hodge theory on orbifolds (see e.g.~\cite{BG}). A special case is when $(X, L)$ is a polarized orbifold  (see \cite{BG} for the definition) with $\alpha = 2\pi c_1(L)$, and $\T$ is covered by a maximal torus $\hat \T \subset {\rm Aut}(X, L)$, see \cite{gauduchon-book}. 
In this situation,  we let $\Pol_{\alpha}:=\Pol_L \subset \tor^*$ be the polytope defined by the chosen lift $\hat \T$ of $\T$ to ${\rm Aut}_{\rm r}(X)$ as follows: 
for any $\hat\T$-invariant hermitian metric $h_{\omega}$ on $\pi_L: L \to X$ with curvature a $\T$-invariant K\"ahler metric $\omega \in 2\pi c_1(L)$, and any section $s\in \CC^{\infty}(X, L)$,  the expression
\[ 
    \mathcal{L}_{\hat \xi} s - \nabla^h_{\hat \xi} s =  -i\mu_{\omega}^{\xi} s, \qquad \hat \xi \in \hat \tor, \, \xi = (\pi_L)_* \hat \xi \in \tor, \, \mu_{\omega}^{\xi} \in \CC^{\infty}(X, \R),
\]
where $\nabla^{h}$ stands for the Chern connection of $(L, h)$, gives rise to an $\omega$-momentum map $\mu_{\omega} : X \to \tor^*$  for $(X, \omega, \T)$, with image $\Pol_L:= \mu_{\omega}(X)$ independent of the choice of $\omega$ and $h$ (see e.g. \cite{Lahdili_2019}).

\smallskip
Let $v, w\in \CC^{\infty}(\Pol_{\alpha}), \, v >0$ be a pair of weight functions on $\Pol_{\alpha}$.  

\begin{defn}[{\cite{Lahdili_2019}}] 
In the above setup, let $\omega \in \alpha$ be a $\T$-invariant K\"ahler metric on $(X, \T)$ with $\Pol_{\alpha}$-normalized momentum map $\mu_\omega : X \to \Pol_{\alpha}$. 
The $v$-scalar curvature $\Scal_v(\omega)$ of $\omega$ is given by
\[
    \Scal_v(\omega)
    := v(\mu_{\omega}) \Scal(\omega) + 2\Delta_{\omega} v(\mu_{\omega}) + \big\langle g_{\omega}, \mu_{\omega}^*\left(\Hess(v)\right)\big\rangle, 
\]
where $\Scal(\omega)$ is the usual scalar curvature of the Riemannian metric $g_\omega$  associated to $\omega$, $\Delta_{\omega}$ is the Laplace operator of $g_{\omega}$, and  $\langle \cdot, \cdot \rangle$  denotes the contraction between the smooth $\tor^*\otimes \tor^*$-valued function $g_{\omega}$ on $X$ (the restriction of the Riemannian metric $g_{\omega}$  to $\tor \subset \CC^{\infty}(X, TX)$) and the smooth $\tor\otimes \tor$-valued function ${\mu_{\omega}}^*\left(\Hess(v)\right)$ on $X$ (given by the pull-back by $\mu_{\omega}$ of ${\rm Hess}(v) \in \CC^{\infty}(\Pol_{\alpha}, \tor\otimes \tor)$). 
\end{defn}

In \cite{Lahdili_2019, Apostolov_Jubert_Lahdili_2023}, a comprehensive study of the geometric PDE
\begin{equation}\label{weighted-cscK} 
    \Scal_v(\omega) = w(\mu_{\omega})
\end{equation}
is provided.  
A $\T$-invariant K\"ahler metric $\omega \in \alpha$ solving \eqref{weighted-cscK} is called \emph{(v,w)-cscK} metric. 
It is shown that if $(X, \T, \alpha)$ admits a $(v,w)$-cscK metric, then the \emph{$(v, w)$-weighted Futaki invariant} $\Fut_{v,w}$ must vanish. 
The latter is viewed as a linear map from the space ${\rm Aff}(\Pol_{\alpha})\cong \tor \oplus \R$ of affine-linear functions on $\Pol_{\alpha}$ and can be introduced by any $\T$-invariant K\"ahler metric  $\omega \in \alpha$ by 
\[
    \Fut_{v,w} (\ell) : = \int_{X} \left(\Scal_{v}(\omega) - w(\mu_{\omega})\right) \ell(\mu_{\omega}) \, \omega^{[n]}, \qquad \ell(x)\in {\rm Aff}(\Pol_{\alpha}).
\]

\smallskip
The above formalism naturally includes many important classes of special K\"ahler metrics. 
We recall below the definition of an \emph{extremal} K\"ahler metric in the sense of Calabi.

\begin{defn}[Extremal K\"ahler metrics]
An extremal K\"ahler metric $\omega_{\rm ext}$ on $X$ is a K\"ahler metric whose scalar curvature $\Scal(\omega_{\rm ext})$ satisfies that $\omega_{\rm ext}^{-1}(d\Scal(\omega_{\rm ext}))$ is a (real) holomorphic vector field (equivalently, it is a Killing vector field with respect to the Riemannian metric associated to $\omega_{\rm ext}$).
\end{defn}

\begin{lem}\label{l:extremal} Given $(X,\alpha, \T, \Pol_{\alpha})$ as above, the following are equivalent:
\begin{enumerate}
\item[(i)] $X$ admits an extremal K\"ahler metric in the de Rham class $\alpha$;
\item[(ii)] $(X, \T, \alpha, \Pol_{\alpha})$ admits a $\T$-invariant  $(1, \ell_{\rm ext})$-cscK metric, where $\ell_{\rm ext}(x)= \langle \xi_{\rm ext}, x\rangle + c_{\rm ext}$ is the unique affine linear function on $\Pol_{\alpha}$ satisfying $\Fut_{1,\ell_{\rm ext}} \equiv 0$.
\end{enumerate}
\end{lem}

\begin{proof} 
The relation (ii) $\Rightarrow (i)$ is clear as $\Scal_1(\omega)=\Scal(\omega)$ and $\ell_{\rm ext}(\mu_{\omega})$ is a Killing potential for $\omega$. 
To prove (i) $\Rightarrow$ (ii), recall that by a result of Calabi, any extremal K\"ahler metric on $X$ is invariant under a maximal torus in $\Aut_{\rm r}(X)$. 
As any two such maximal tori are conjugated in $\Aut_{\rm r}(X)$, we may assume without loss of generality that $X$ admits a $\T$-invariant extremal K\"ahler metric $\omega_{\rm ext} \in \alpha$. 
The extremal Killing vector field $\xi_{\rm ext}:=-\omega_{\rm ext}^{-1}(d\Scal(\omega_{\rm ext}))$  is then $\T$-invariant and, by the maximality of $\T$, it belongs to the Lie algebra $\tor$ of $\T$.  
It then follows that $\Scal(\omega_{\rm ext}) = \ell(\mu_{\omega_{\rm ext}})$ where $\ell(x)= \langle \xi, x \rangle + c$ is an affine-linear function on $\tor$, i.e. $\omega_{\rm ext}$ is $(1, \ell)$-cscK and thus $\Fut_{1, \ell} \equiv 0$. 
However, there exists a unique $\ell=\ell_{\rm ext}$ with the latter property, given by the $L^2(X, \omega)$ projection of the scalar curvature of any $\T$-invariant K\"ahler metric  $\omega \in \alpha$ to the pullback of ${\rm Aff}(\tor^*)$ via $\mu_{\omega}$.
\end{proof}

\subsection{Sasaki structures with constant transversal scalar curvature as $(v,w)$-weighted cscK metrics} 
In this section, we return to the Sasaki setup and assume that $(\Sa, \hat \xi, \eta^{\hat \xi}\Ds, J)$ is a compact smooth Sasaki manifold with respect to a (not necessarily quasi-regular) Sasaki--Reeb vector field $\hat \xi \in \hat \tor_+(\Sa)$ and induced transversal K\"ahler structure $\omega_{\hat \xi} : = (d\eta^{\hat \xi})_{|_{\Ds}}$.

\smallskip
In the sequel, we shall fix once for all a quasi-regular Sasaki--Reeb vector field $\hat \chi \in \hat\tor_+(\Sa)$ and denote by $(X, \omega, \T)$ the K\"ahler orbifold quotient of $(\Sa, \hat\chi, \eta^{\hat \chi}, \Ds, J)$ by the $\Sph^1_{\hat\chi}$ subgroup of $\hat \T$ generated by $\hat\chi$, where $\T := \hat \T/\Sph^1_{\hat\chi}$ is the induced torus action on $X$ and $\omega$ is the induced K\"ahler structure by the transverse K\"ahler structure $\omega_{\hat \chi}= (d\eta^{\hat \chi})_{|_{\Ds}}$ on $(\Sa, \hat \chi, \Ds, J)$. 
Equivalently, one can see $X$ as a polarized orbifold, then denoted $(X, L)$,  corresponding to the quotient of the polarized K\"ahler cone $(Y, \hat \xi, \hat \omega)$  associated to $(\Sa, \hat \chi, \Ds, J)$ with $Y:= (L^{-1})^{\times}$ and $\omega \in 2\pi c_1(L)$.
It is well-known that the map
\[
    \hat \tor \ni \hat \zeta \to \nu^{\hat \zeta}:=\eta_{\hat \chi}(\hat \zeta) \in \CC^{\infty}(\Sa, \R)
\]
gives rise to a contact $\hat \T$-momentum map whose image $\hat \Sigma:=\nu (\Sa) \subset \hat \tor^*$ is a convex polyhedral cone. 
We further introduce the subspace
\[
    \Pol_{\hat \xi} := \{\ell\in \hat \Sigma \, | \,  \langle \hat \zeta, \hat \chi\rangle =1 \}
\]
A key fact in this theory is the identifications  (see e.g. \cite[Lemma~1.15]{ACL} and the references therein):
\begin{equation}\label{polytopes} 
    \Pol_{\hat \xi} \cong \Pol_L, \qquad \tor_+(\Sa) \cong \Pol_L^* := \{ \ell(x)=\langle \zeta, x\rangle +1 \in {\rm Aff}(\Pol_L) \, | \,  \ell(x) \geq 0  \, \text{on} \, \Pol_L \}. 
\end{equation}
With this understood, an observation from \cite{AC} gives the following: 

\begin{prop}\label{p:cscS-weighted} 
In the above setup, $(\Sa, \hat\xi, \Ds, J)$ has constant transversal scalar curvature (equal to $a$) if and only if the K\"ahler orbifold $(X, \omega, \T, \Pol_L)$ is $(v, w)$-cscK with
\[
    v(x):= \ell(x)^{-(n+1)},  \qquad w(x):= a \ell(x)^{-(n+2)},   
\]
where $\ell(x)=\langle \xi, x \rangle +1, \, \xi := [\hat \xi] \in \hat \tor/\langle \hat \chi \rangle$ is a positive affine-linear function on $\Pol_L$ corresponding to $\hat\xi$ via \eqref{polytopes}.
\end{prop}

Putting Proposition~\ref{sasaki-identification}, Lemma~\ref{l:cscS-cones} and Proposition~\ref{p:cscS-weighted} together, we obtain the following correspondence:

\begin{cor}\label{c:tool} 
Let $(Y= (L^{-1})^{\times}, \hat \chi, \hat \omega_0)$ be  quasi-regular polarized K\"ahler cone over a polarized K\"ahler orbifold $(X= Y/\C^*_{\hat \chi}, L)$, corresponding to quasi-regular Sasaki manifold $(\Sa, \hat \chi, \eta_0^{\hat \chi}, \Ds_0, J_0)$. 
Let  $\hat\T \subset \Aut(X, L)$ be a maximal torus with Lie algebra $ \hat \tor$ containing the quasi-regular polarization $\hat\chi$ and covering a maximal torus $\T \subset \Aut_{\rm r}(X)$. 
Denote by $\hat\tor_+(Y) = \hat\tor(\Sa)$ the corresponding Sasaki--Reeb cone. 
Then the following conditions are equivalent:
\begin{enumerate}
    \item[(i)] $\exists$ a (ray of)  $\hat \xi \in \hat\tor_+(\Sa)$ and $\varphi \in \Xi^{\hat\T}_{\hat\xi, \eta_0^{\hat\xi}, J^{\hat\xi}}(\Sa)$ such that $\Sa$ admits a Sasaki structure $(\hat\xi, \eta^{\hat\xi}_{\varphi}, \Ds_{\varphi}, J_{\varphi})$ of constant transversal scalar curvature equal to $a$.
    \item[(ii)] $\exists$ $\xi \in \tor:= {\rm Lie}(\T)$ with $\ell(x):= \langle \xi, x\rangle +1 >0$ on $\Pol_L$  and a constant $a$ such that $(X, \T)$ admits a $(v,w)$-cscK metric in $\alpha = 2\pi c_1(L)$ with
    \[ 
        v(x)= \ell(x)^{-(n+1)}, \qquad w(x)=a\ell(x)^{-(n+2)}.
    \]
\end{enumerate}
Furthermore, if the constant $a$ in the above equivalent conditions is positive, each of these conditions is equivalent to the statement
\begin{enumerate}
    \item[(iii)] $\exists$ $\hat \xi_0 \in \hat\tor_+(Y)$ such that $Y$ admits a scalar-flat K\"ahler cone metric polarized by $\hat \xi_0$. Furthermore, $\hat \xi_0$ belongs to the ray of Sasaki--Reeb vector fields $\xi$  satisfying $(i)$.
\end{enumerate}
\end{cor}

\subsection{A LeBrun--Simanca stability theorem}
We state below a LeBrun--Simanca type stability result  established in the smooth non-singular case in \cite{Apostolov_Lahdili_Legendre_2024} with arguments which carry over the orbifold case:

\begin{lem}[{\cite{Apostolov_Lahdili_Legendre_2024}}]\label{l:(v,w)-LeBrun-Simanca}  
Let $\{(v_t, w_t), \, \, v_t, w_t \in \CC^{\infty}(\Pol_{\alpha}), \, v_t>0,  \, \,  t\in U \subset \R^k\}$ be a finite dimensional smooth family of weight  functions on $\Pol_{\alpha}$, parametrized by a neighbourhood $U$ of $0$ and satisfying  $\Fut_{v_t, w_t}\equiv 0$ for all $t\in U$. Suppose $\omega_0$ is a $\T$-invariant K\"ahler metric in $\alpha$ such that
\[ \Scal_{v_0}(\omega_0)= w_0(\mu_{\omega_0}).\] 
Then, there exists $\epsilon >0$ and a differentiable family of smooth $\T$-invariant K\"ahler metrics $\omega_t \in \alpha$, such that for $\|t\| < \epsilon$
\[
    \Scal_{v_t}(\omega_t)=w_t(\mu_{\omega_t}),
\]
where $\mu_{\omega_t}$ is the $\omega_t$-momentum map of $\T$ normalized by $\mu_{\omega_t}(X)=\Pol_{\alpha}$.
\end{lem}

\section{A family of Einstein--Hilbert functionals}
In this section, we shall construct Einstein--Hilbert functionals for our approximate weights. 
Before entering the Einstein--Hilbert functionals, we recall the definition of the weighted Laplacian. 
Let $(X, \af, \T)$ be as in Section~\ref{subsec:vwcscK} and let $v > 0$ be a positive smooth function defined on $\Pol_\af$. 
The weighted Laplacian is defined as follows: 

\begin{defn}[{\cite{Lahdili_2019}}]\label{defn:wLap}
Let $\om \in \af$ be a $\T$-invariant K\"ahler metric with $\Pol_{\alpha}$-normalized momentum map $\mu_{\omega}$. 
The $v$-weighted Laplacian acts on smooth functions $f \in \CC^\infty(X)$ and is defined by 
\[
    \Dt_{\om,v} (f) = \frac{1}{v(\mu_\om)} \dt_\om (v(\mu_\om) d f),
\]
where $\dt_\om$ is the formal adjoint of the Levi--Civita connection associated to $\om$. 
\end{defn}

We remark that $\Dt_{\om,v}$ is a second-order elliptic operator and it is self-adjoint with respect to the weighted volume form $v(\mu_\om) \om^{[n]}$.

\smallskip
We now present the construction of the Einstein--Hilbert functionals that play the main role in this article. 
Consider the following function
\[
    F_k(a,t) = \exp\left(\left(\frac{1-ka}{a}\right) \log (1+at)\right) = (1+at)^{\frac{1-ka}{a}}
\]
where $(a,t) \in \BR^2$, $k \in \BN$ and $0< 1+at < 2$. 
Note that $F_k$ admits an analytic extension at $a = 0$ with $F_k(0,t) = e^t$ for all $k \in \BN$. 
We then define the following functionals
\begin{defn}
Given a $\T$-invariant K\"ahler metric $\om \in \af$ with $\Pol_{\alpha}$-normalized momentum map $\mu_{\omega}$, $a \in \BR \setminus \{0\}$, and an affine-linear function $l(x)$ defined on $\mathfrak{t}^*$ and satisfying $0 < 1 + a\ell(x) < 2$ for all $x \in \Pol_{\alpha}$, we consider the following Einstein--Hilbert functional  
\begin{equation}\label{eq:EH_defn}
\begin{split}
    \EH_{\ld,a,b}(l) 
    &:= \frac{\BFS_a(l)}{\BFV_a(l)} + \ld \left(- \frac{\int_X F_0(a,l(\mu_\om)) \om^{[n]}}{a \BFV_a(l)} + \frac{1}{a}\right) + \frac{b}{\BFV_a(l)}\\
    &= \frac{1}{\BFV_a(l)} \left(\BFS_a(l) - \frac{\ld}{a} \int_X F_0(a,l) \om^{[n]}\right) + \frac{\ld}{a} + \frac{b}{\BFV_a(l)}
\end{split}
\end{equation}
where 
\[
    \BFS_a(l) 
    := \int_X  \Scal_{F_1(a,l(x))} (\om) \cdot (1+al(\mu_\om)) \om^{[n]} 
\]
and
\[
    \BFV_a(l) 
    := \left(\int_X F_1(a,l(\mu_\om)) \om^{[n]}\right)^{\frac{1}{1-a}}.
\]
\end{defn}

\begin{rmk}
The functionals $\BFS_a(l)$, $\BFV_a(l)$ and $\EH_{\ld,a}(l)$ do not depend on the K\"ahler metrics in the K\"ahler class $\af$, see e.g. \cite[Lemma 2]{Lahdili_2019}.
\end{rmk}

\begin{rmk}
The leading term in our construction \eqref{eq:EH_defn} is inspired by \cite{BHL, FO,  Lahdili_Legendre_Scarpa_2023}. 
The additional twisted terms are introduced so that the Euler--Lagrange equation of $\EH_{\ld,a,b}$ recovers the desired weighted Futaki invariant for approximate (almost) scalar-flat cone weights (cf. \eqref{eq:EH_1stvar} and \eqref{eq:Futa_defn}), and also to provide some freedom to adjust the vanishing of the weighted Futaki invariant (cf. \eqref{eq:Fut_adjust}). 
\end{rmk}

In the following,  we  shall  prove that the functionals $\EH_{\ld,a}(\bullet)$ converges to $\EH_{\ld,0}(\bullet)$ in $\CC^2_{\loc}$ as $a$ tends to $0$. 
The precise definition of $\EH_{\ld,0}(\bullet)$ is given in Lemma~\ref{lem:EH_C0conv}.

\subsection{Locally uniform convergence of the Einstein--Hilbert functional}
\begin{lem}\label{lem:EH_C0conv}
The Einstein--Hilbert functional $\EH_{\ld,a,b}(l)$ converges locally uniformly to $\EH_{\ld,0,b}(l)$ when $a \to 0$, where
\[
    \EH_{\ld,0,b}(l) := \frac{\BFS_0(l)}{\BFV_0(l)} + \ld \left(-\frac{\int_X l e^l \om^{[n]}}{\BFV_0(l)} + \log \BFV_0(l)\right) 
    + \frac{b}{\BFV_0(l)}.
\]
\end{lem}

\begin{rmk}
In Inoue's work \cite{Inoue_2022}, he has considered the following functional: 
\begin{align*}
	\log \Vol^\ld(\xi) 
    &= \frac{1}{\int_X e^{\langle \xi, \mu \rangle} \om^n} \int_X e^{\langle \xi, \mu\rangle} \left(\Scal(\om) + \Dt_\om \langle\xi,\mu\rangle -\ld \langle\xi,\mu\rangle\right) \om^n 
    + \ld \log \left(\int_X e^{\langle \xi, \mu \rangle} \om^n \right)\\
    &= \frac{1}{\int_X e^{\langle \xi, \mu \rangle} \om^n} \int_X \Scal_{e^{\langle\xi,\bullet\rangle}}(\om) + \ld \left(\frac{- \int_X \langle\xi,\mu\rangle e^{\langle\xi,\mu\rangle} \om^n}{\int_X e^{\langle \xi, \mu \rangle} \om^n} + \log \left(\int_X e^{\langle\xi,\mu \rangle}\om^n\right)\right).
\end{align*}
For a general weight $v$, we introduce the following generalization 
\begin{equation}\label{eq:Inoue_func_general}
	\log \Vol^\ld(v) 
    = \frac{\int_X \Scal_v(\om) \om^n}{\int_{\mathrm{P}_X} v \, \dd \mu_{\mathrm{DH}}}
	+ \ld \left(\frac{- \int_{\mathrm{P}_X} \log v \cdot v \, \dd \mu_{\mathrm{DH}}}{\int_{\mathrm{P}_X} v \, \dd \mu_{\mathrm{DH}}}
    + \log \left(\int_{\mathrm{P}_X} v \, \dd \mu_{\mathrm{DH}}\right)\right),
\end{equation}
where $\mu_{\mathrm{DH}} = (\mu_\om)_\ast \om^{[n]}$ is the Duistermaat--Heckman measure.
The functional $\EH_{\ld,0,0}(\bullet)$ in Lemma~\ref{lem:EH_C0conv} coincides to the above log volume functional $\log \Vol^\ld(e^{\bullet})$.
\end{rmk}

\begin{proof}
For simplicity, we denote by $l = l(\mu_\om)$. 
It is not hard to see that  
\[
    \BFS_a(l) \xrightarrow[a \to 0]{} \int_X \Scal_{e^l}(\om) \om^{[n]}
    \quad\text{and}\quad
    \BFV_a(l) \xrightarrow[a \to 0]{}
    \BFV_0(l) = \int_X e^l \om^{[n]};
\]
therefore, it suffices to verify 
\begin{equation}\label{eq:EH_conv_1}
    \frac{1}{a \BFV_a(l)} \left(\BFV_a(l) - \int_X F_0(a,l) \om^{[n]}\right) 
    \xrightarrow[a \to 0]{} \frac{1}{\BFV_0(l)} \left(-\int_X l e^l \om^{[n]} + \BFV_0(l) \log \BFV_0(l)\right)
\end{equation}
To do so, we shall compute the first-order Taylor expansion. 
Since $\log (1+al) = al - \frac{a^2 l^2}{2} + O(a^3)$, we have 
\[
    \frac{1-ka}{a} \log (1+al) = l + a \left(-\frac{l^2}{2} - kl\right) + O(a^3)
\]
and thus 
\begin{equation}\label{eq:Fk_exp}
    F_k(a,l) = e^l \left(1 + a \left(- l^2/2 - kl\right) + O(a^2)\right). 
\end{equation}
From \eqref{eq:Fk_exp}, we have 
\begin{equation}\label{eq:Fk_int_exp}
\begin{split}
    \int_X F_k(a,l) \om^{[n]} 
    &= \int_X e^l (1+a(-l^2/2 - kl) + O(a^2)) \om^{[n]} \\
    &= \int_X e^l \om^{[n]} + a \int_X e^l(-l^2/2 - kl) \om^{[n]} + O(a^2)
\end{split}
\end{equation}
and therefore
\begin{align*}
    \log \left(\int_X F_1(a,l) \om^{[n]}\right) 
    &= \log \BFV_0(l) + \log \left(1 + a \frac{\int_X e^l(-l^2/2 - l) \om^{[n]}}{\BFV_0(l)} + O(a^2)\right)\\
    &= \log \BFV_0(l) + a \frac{\int_X e^l(-l^2/2 - l)}{\BFV_0(l)}+ O(a^2).
\end{align*}
Thus, one can derive that 
\begin{align*}
    \frac{1}{1-a} \log \left(\int_X F_1(a,l) \om^{[n]}\right) 
    &= (1+a+O(a^2)) \left(\log \BFV_0(l) + a \frac{\int_X e^l (-l^2/2 - l) \om^{[n]}}{\BFV_0(l)} + O(a^2)\right)\\
    &= \log \BFV_0(l) + a \left(\log \BFV_0(l) + \frac{\int_X e^l(-l^2/2 - l) \om^{[n]}}{\BFV_0(l)}\right) + O(a^2)
\end{align*}
and 
\begin{equation}\label{eq:Va_exp}
\begin{split}
    \BFV_a(l) 
    &= \exp \left(\frac{1}{1-a} \log \left(\int_X F_1(a,l) \om^{[n]}\right) \right) \\
    &= \exp \left(\log \BFV_0(l) + a \left(\log \BFV_0(l) + \frac{\int_X e^l(-l^2/2 - l) \om^{[n]}}{\BFV_0(l)}\right) + O(a^2)\right)\\
    &= \BFV_0(l) \exp \left(a \left(\log \BFV_0(l) + \frac{\int_X e^l (-l^2/2 - l) \om^{[n]}}{\BFV_0(l)}\right) + O(a^2)\right)\\
    &= \BFV_0(l) + a \left( \BFV_0(l) \log \BFV_0(l) + \int_X e^l(-l^2/2 - l) \om^{[n]}\right) + O(a^2).
\end{split}
\end{equation}
Combining \eqref{eq:Fk_int_exp} and \eqref{eq:Va_exp}, we infer that  
\begin{align*}
    \frac{1}{a} \left(\BFV_a(l) - \int_X F_0(a,l) \om^{[n]}\right)  
    &= \frac{1}{a} \bigg[\BFV_0(l) + a \left( \BFV_0(l) \log \BFV_0(l) + \int_X e^l(-l^2/2 - l) \om^{[n]}\right)\\
    &\qquad\qquad- \BFV_0(l) - a \int_X e^l(-l^2/2) \om^{[n]}\bigg] + O(a)\\
    &= - \int_X l e^l \om^{[n]} + \BFV_0(l) \log \BFV_0(l) + O(a).
\end{align*}
This shows \eqref{eq:EH_conv_1}, and all in all, we obtain that the following locally uniform convergence of the Einstein--Hilbert functional
\[
    \EH_{\ld,a,b}(l) 
    \xrightarrow[a \to 0]{}
    \frac{\int_X \Scal_{e^l}(\om) \om^{[n]}}{\BFV_0(l)} + \ld \left(-\frac{\int_X l e^l \om^{[n]}}{\BFV_0(l)} + \log \BFV_0(l)\right) + \frac{b}{\BFV_0(l)}.
\]
This completes the proof.
\end{proof}

\subsection{Variation of $\EH_{\ld,a}$}
We shall now calculate the first-order variation of $\EH_{\ld,a}$. 
Note that 
\[
    (\pl_t^j F_k)(a,t) = \left(\prod_{i=1}^j (1-(i+k-1)a) \right) \cdot (1+at)^{\frac{1-(j+k)a}{a}}.
\]
We now calculate the linearization of $l \mapsto \Scal_{F_1(a,l(\bullet))}(\om)$. 

\begin{lem}\label{lem:Scal_lin}
The linearized operator of $l \mapsto \Scal_{F_1(a,l(\bullet))}(\om)$ has the following form
\[
    L_{\ld,a,l}(\dot{l}) 
    = (1-a) \left(\Scal_{F_1(a,l)}(\om) \cdot (1+al)^{-1} \dot{l} + 2 F_1(a,l) \Dt_{\om,F_0(a,l)} \left((1+al)^{-1} \dot{l}\right)\right) 
\]
and as $a \to 0$, $L_{\ld,a,l}(\dot{l})$ converges locally uniformly in $l$ towards
\[
    L_{\ld,0,l}(\dot{l})
    = \Scal_{e^l}(\om) \dot{l} + 2 e^l \Dt_{\om,e^l} (\dot{l}).
\]
Here $\Dt_{\om,F_0(a,l)}$ and $\Dt_{\om,e^{l}}$ are weighted Laplacians given in Definition~\ref{defn:wLap}.
\end{lem}

\begin{proof}
Again, for simplicity, we shall denote by $l = l(\mu_\om)$ and $\dot{l} = \dot{l}(\mu_\om)$. 
Recall that the weighted scalar curvature has the following form 
\[
    \Scal_v(\om) 
    = v(\mu_\om) \Scal(\om) 
    + 2 \Dt_\om v(\mu_\om)
    + \langle g_\om, \mu_\om^\ast \Hess(v)\rangle;
\]
therefore, we have
\begin{align*}
    \Scal_{F_1(a,l(\mu_\om))}(\om) 
    &= F_1(a,l) \Scal(\om) + 2 \Dt_\om (F_1(a, l)) + \langle g_\om, \mu_\om^\ast \Hess(F_1(a,l)) \rangle \\
    &= F_1(a,l) \Scal(\om) + 2 (\pl_t F_1)(a,l) \Dt_\om l - (\pl_t^2 F_1)(a,l) |d l|_\om^2\\
    &= (1+al)^{\frac{1-a}{a}} \Scal(\om) + 2 (1-a) (1+al)^{\frac{1-2a}{a}} \Dt_\om l\\ 
    &\qquad- (1-a)(1-2a) (1+al)^{\frac{1-3a}{a}} |dl|_\om. 
\end{align*}
The linearized operator can then be computed as follows
\begingroup\allowdisplaybreaks
\begin{align*}
    L_{\ld,a,l}(\dot{l}) 
    &= (1-a) (1+al)^{\frac{1-2a}{a}} \dot{l} \Scal(\om)
    + 2 (1-a)(1-2a) (1+al)^{\frac{1-3a}{a}} \dot{l} \Dt_\om l \\
    &\quad + 2 (1-a) (1+al)^{\frac{1-2a}{a}} (\Dt_\om \dot{l})
    - (1-a)(1-2a)(1-3a) (1+al)^{\frac{1-4a}{a}} \dot{l} |d l|_\om^2 \\
    &\quad - 2 (1-a)(1-2a) (1+al)^{\frac{1-3a}{a}} \langle d l, d \dot{l}\rangle_\om\\
    &= \left(\frac{1-a}{1+al} F_1(a,l) \Scal(\om) + 2 \cdot \frac{1-2a}{1+al} (\pl_t F_1)(a,l) \Dt_\om l 
    - \frac{1-3a}{1+al} (\pl_t^2 F_1)(a,l) |d l|_\om^2\right) \dot{l} \\
    &\quad + 2 (1-a) (1+al)^{\frac{1-2a}{a}} (\Dt_\om \dot{l}) 
    - 2 (1-a)(1-2a) (1+al)^{\frac{1-3a}{a}} \langle d l, d \dot{l}\rangle_\om\\
    &= \frac{1-a}{1+al} \cdot \Scal_{F_1(a,l)}(\om) \dot{l}
    - \frac{2a}{1+al} \cdot \left((\pl_t F_1) \dot{l} \Dt_\om l - (\pl_t^2 F_1) \dot{l} |dl|_\om^2 \right) \\
    &\quad + 2 (1-a) (1+al)^{\frac{1-2a}{a}} (\Dt_\om \dot{l}) 
    - 2 (1-a)(1-2a) (1+al)^{\frac{1-3a}{a}} \langle d l, d \dot{l}\rangle_\om\\
    &= \frac{1-a}{1+al} \cdot \Scal_{F_1(a,l)}(\om) \dot{l} \\
    &\quad- \frac{2a}{1+al} \cdot \left((1-a) (1+al)^{\frac{1-2a}{a}} \dot{l} \Dt_\om l - (1-a)(1-2a) (1+al)^{\frac{1-3a}{a}} \dot{l} |dl|_\om^2 \right) \\
    &\quad + 2 (1-a) (1+al)^{\frac{1-2a}{a}} (\Dt_\om \dot{l}) 
    - 2 (1-a)(1-2a) (1+al)^{\frac{1-3a}{a}} \langle d l, d \dot{l}\rangle_\om\\
    &= \frac{1-a}{1+al} \cdot \Scal_{F_1(a,l)}(\om) \dot{l}\\ 
    &\quad- 2a\left((1-a) (1+al)^{\frac{1-3a}{a}} \dot{l} \Dt_\om l - (1-a)(1-2a) (1+al)^{\frac{1-4a}{a}} \dot{l} |dl|_\om^2 \right) \\
    &\quad + 2 (1-a) (1+al)^{\frac{1-2a}{a}} (\Dt_\om \dot{l}) 
    - 2 (1-a)(1-2a) (1+al)^{\frac{1-3a}{a}} \langle d l, d \dot{l}\rangle_\om\\
    &= \left(\frac{1-a}{1+al}\right) \Scal_{F_1(a,l)}(\om) \dot{l} 
    - 2a(1-a) (1+al)^{-1} \dt_\om((1+al)^{\frac{1-2a}{a}} dl) \dot{l} \\
    &\quad+ 2(1-a)\dt_\om \left((1+al)^{\frac{1-2a}{a}} d \dot{l}\right).
\end{align*} 
\endgroup
We now pay attention to the last term $\dt_\om ((1+al)^{\frac{1-2a}{a}} d \dot{l})$ in the above expression. 
We have 
\begingroup\allowdisplaybreaks
\begin{align*}
    &\dt_\om \left((1+al)^{\frac{1-2a}{a}} d \dot{l}\right) 
    = \dt_\om \left((1+al)^{\frac{1-2a}{a}} d \left((1+al) (1+al)^{-1} \dot{l}\right)\right)\\
    &\qquad= \dt_\om \left((1+al)^{\frac{1-a}{a}} d \left((1+al)^{-1} \dot{l}\right)\right) + \dt_\om \left((1+al)^{\frac{1-2a}{a}} a(1+al)^{-1} \dot{l} dl\right)\\
    &\qquad= \dt_\om \left((1+al)^{\frac{1-a}{a}} d \left((1+al)^{-1} \dot{l}\right)\right) 
    + a(1+al)^{-1} \dot{l} \cdot \dt_\om \left((1+al)^{\frac{1-2a}{a}}  dl\right)\\
    &\qquad\quad - a (1+al)^{\frac{1-2a}{a}} \langle d l, d((1+al)^{-1} \dot{l}) \rangle_\om\\
    &\qquad= (1+al)^{\frac{1-a}{a}} \Dt_\om\left((1+al)^{-1} \dot{l}\right) - (1-a) (1+al)^{\frac{1-2a}{a}} \langle dl, d((1+al)^{-1} \dot{l})\rangle_\om\\
    &\qquad\quad + a(1+al)^{-1} \dot{l} \cdot \dt_\om \left((1+al)^{\frac{1-2a}{a}}  dl\right) - a (1+al)^{\frac{1-2a}{a}} \langle d l, d((1+al)^{-1} \dot{l}) \rangle_\om\\
    &\qquad= (1+al)^{\frac{1-a}{a}} \Dt_\om\left((1+al)^{-1} \dot{l}\right) 
    - (1+al)^{\frac{1-2a}{a}} \langle dl, d((1+al)^{-1} \dot{l})\rangle_\om\\
    &\qquad\quad + a(1+al)^{-1} \dot{l} \cdot \dt_\om \left((1+al)^{\frac{1-2a}{a}}  dl\right) \\
    &\qquad= (1+al)^{-1}\left[ (1+al)^{\frac{1}{a}} \Dt_\om\left((1+al)^{-1} \dot{l}\right) 
    - (1+al)^{\frac{1-a}{a}} \langle dl, d((1+al)^{-1} \dot{l})\rangle_\om\right]\\
    &\qquad\quad + a(1+al)^{-1} \dot{l} \cdot \dt_\om \left((1+al)^{\frac{1-2a}{a}}  dl\right) \\
    &\qquad= (1+al)^{-1} \dt_\om\left( (1+al)^{\frac{1}{a}} d\left((1+al)^{-1} \dot{l}\right)\right) + a(1+al)^{-1} \dot{l} \cdot \dt_\om \left((1+al)^{\frac{1-2a}{a}}  dl\right).
\end{align*}
\endgroup
Combining the above two equations, we obtain the following expression of the linearized operator as desired
\begin{align*}
    L_{\ld,a,l}(\dot{l}) &= \frac{1-a}{1+al} \cdot \Scal_{F_1(a,l)}(\om) \dot{l} 
    - 2a(1-a) (1+al)^{-1} \dot{l} \cdot \dt_\om((1+al)^{\frac{1-2a}{a}} dl)\\
    &\quad + 2(1-a)\bigg[
    (1+al)^{-1} \dt_\om\left( (1+al)^{\frac{1}{a}} d\left((1+al)^{-1} \dot{l}\right)\right) \\
    &\qquad\qquad\qquad\qquad + a(1+al)^{-1} \dot{l} \cdot \dt_\om \left((1+al)^{\frac{1-2a}{a}}  dl\right)
    \bigg] \\
    &= (1-a)\left\{\Scal_{F_1(a,l)}(\om) (1+al)^{-1}  \dot{l} + 2 (1+al)^{-1} \dt_\om ((1+al)^{\frac{1}{a}} d(1+al)^{-1} \dot{l})\right\}\\
    &= (1-a) \left\{\Scal_{F_1(a,l)}(\om) (1+al)^{-1}  \dot{l} + 2 (1+al)^{\frac{1-a}{a}}\Dt_{F_0(a,l)} \left((1+al)^{-1} \dot{l}\right)\right\}.
\end{align*}
This finishes the proof.
\end{proof}

We then calculate the first variational formula of the Einstein--Hilbert functional. 
\begin{lem}
The first variational formulas of $\BFS_a$, $\BFV_a$, and $\EH_{\ld,a,b}$ have the following forms: 
\begin{equation}\label{eq:Sa_1stvar}
    (\dd \BFS_a)_l 
    = \int_X \Scal_{F_1(a,l)}(\om) \dot{l} \om^{[n]}, 
\end{equation}
\begin{equation}\label{eq:Va_1stvar}
    (\dd \BFV_a)_l 
    = \BFV_a(l)^a \int_X F_2(a,l) \dot{l} \om^{[n]},
\end{equation}
\begin{equation}\label{eq:EH_1stvar}
    (\dd \EH_{\ld,a,b})_l(\dot{l}) 
    = \frac{\Fut_{\ld,a,b,l}(\dot{l})}{\BFV_a(l)},  
\end{equation}
where
\begin{equation}\label{eq:Futa_defn}
\begin{split}
    \Fut_{\ld,a,b,l}(\dot{l}) 
    &:= \int_X \left(\Scal_{F_1(a,l)}(\om) - \frac{\ld}{a} (F_1(a,l) - F_2(a,l))\right) \dot{l} \om^{[n]}\\
    &\qquad - \int_X \left(c_{\ld,a}(l) + \frac{b}{\BFV_a(l)^{1-a}}\right) F_2(a,l) \dot{l} \om^{[n]}
\end{split}
\end{equation} 
with 
\begin{equation}\label{eq:ca_defn}
    c_{\ld,a}(l) 
    := \BFV_a(l)^a \EH_{\ld,a,0}(l) - \frac{\ld}{a}(\BFV_a(l)^a - 1). 
\end{equation}
\end{lem}

\begin{proof}
By the previous Lemma~\ref{lem:Scal_lin}, we have 
\[
    L_{\ld,a,l}(\dot{l}) 
    = (1-a) \left(\Scal_{F_1(a,l)}(\om) \cdot (1+al)^{-1} \dot{l} + 2 F_1(a,l) \Dt_{\om,F_0(a,l)} \left((1+al)^{-1} \dot{l}\right)\right).
\]
Hence, one can derive that
\begin{align*}
    (\dd \BFS_a)_l (\dot{l}) 
    &= (1-a) \int_X \Scal_{F_1(a,l)}(\om) \dot{l} \om^{[n]} 
    + 2(1-a) \int (1+al)^{\frac{1}{a}} \Dt_{F_0(a,l)} \left((1+al)^{-1} \dot{l} \right) \om^{[n]}\\
    &\qquad + \int_X \Scal_{F_1(a,l)}(\om) a \dot{l} \om^{[n]}\\
    &= \int_X \Scal_{F_1(a,l)}(\om) \dot{l} \om^{[n]} 
\end{align*}
and this shows \eqref{eq:Sa_1stvar}. 
Then we calculate the first variation of $\BFV_a$: 
\begin{align*}
    (\dd \BFV_a)_l (\dot{l}) 
    &= \frac{1}{1-a} \left(\int_X F_1(a,l) \om^{[n]}\right)^{\frac{a}{1-a}} \int_X (1-a) (1+al)^{\frac{1-2a}{a}} \dot{l} \om^{[n]}\\
    &= \BFV_a(l)^a \int_X (1+al)^{\frac{1-2a}{a}} a \dot{l} \om^{[n]}\\
    &= \BFV_a(l)^a \int_X F_2(a,l) \dot{l} \om^{[n]}. 
\end{align*}
To get the first variational formula of the Einstein--Hilbert functional, we would need the following 
\begin{equation}\label{eq:Sald_1stvar}
\begin{split}
    \left(\dd \left(\BFS_a + \frac{\ld}{-a} \int_X F_0(a,l) \om^{[n]}\right)\right)_l(\dot{l}) 
    &= \int_X \left(\Scal_{F_1(a,l)}(\om) + \frac{\ld}{-a} F_1(a,l)\right) \dot{l} \om^{[n]}
\end{split}
\end{equation}
Combining \eqref{eq:Sald_1stvar} and \eqref{eq:Va_1stvar}, we obtain that 
\begin{align*}
    &(\dd \EH_{\ld,a,b})_l(\dot{l}) \\
    &\quad= \frac{1}{\BFV_a(l)} \int_X \left(\Scal_{F_1(a,l)}(\om) + \frac{\ld}{-a} F_1(a,l)\right) \dot{l} \om^{[n]}\\
    &\qquad - \frac{1}{\BFV_a(l)^2} \left(\BFS_a(l) + \frac{\ld}{-a} \int_X F_0(a,l) \om^{[n]}\right) \cdot \BFV_a(l)^a \int_X F_2(a,l) \dot{l} \om^{[n]} \\
    &\qquad -\frac{b \BFV_a(l)^a \int_X F_2(a,l) \dot{l} \om^{[n]}}{\BFV_a(l)^2}\\
    &\quad= \frac{1}{\BFV_a(l)} \int_X \left(\Scal_{F_1(a,l)}(\om) - \frac{\ld}{a} F_1(a,l) - \left[c_{\ld,a}'(l) + \frac{b}{\BFV_a(l)^{1-a}}\right] F_2(a,l)\right) \dot{l} \om^{[n]}\\
    &\quad= \frac{1}{\BFV_a(l)} \int_X \left(\Scal_{F_1(a,l)}(\om) - \frac{\ld}{a} [F_1(a,l) - F_2(a,l)] - \left[c_{\ld,a}(l) + \frac{b}{\BFV_a(l)^{1-a}} 
    \right] F_2(a,l)\right) \dot{l} \om^{[n]}
\end{align*} 
and this shows \eqref{eq:EH_1stvar}.
Here we set 
\begin{equation}\label{eq:c'_defn}
    c_{\ld,a}'(l) 
    := \frac{1}{\BFV_a(l)^{1-a}} \left(\BFS_a(l) - \frac{\ld}{a}\int_X F_0(a,l) \om^{[n]}\right) 
    = \BFV_a(l)^a \left(\EH_{\ld,a}(l) - \frac{\ld}{a}\right).
\end{equation}
and 
\begingroup\allowdisplaybreaks
\begin{align*}
    c_{\ld,a}(l) 
    &= c_{\ld,a}'(l) + \frac{\ld}{a} 
    = \frac{1}{\BFV_a(l)^{1-a}} \left(\BFS_a(l) -\frac{\ld}{a} \left(\int_X F_0(a,l) \om^{[n]} - \BFV_a(l)^{1-a}\right)\right)\\
    &= \BFV_a(l)^a \EH_{\ld,a}(l) - \frac{\ld}{a} (\BFV_a(l)^a - 1).
\end{align*}
\endgroup
\end{proof}

\begin{lem}\label{lem:EH0_1stvar}
The first variational formula of $\EH_{\ld,0,b}$ is 
\[
    (\dd\EH_{\ld,0,b})_l(\dot{l}) 
    = \frac{\Fut_{\ld,0,b,l}(\dot{l})}{\BFV_0(l)},
\]
where 
\begin{equation}\label{eq:Fut0_defn}
    \Fut_{\ld,0,b,l}(\dot{l}) 
    := \int_X \left(\Scal_{e^l}(\om)  - \ld l e^l - e^l \left[c_{\ld,0}(l) + \frac{b}{\BFV_0(l)}\right]\right) \dot{l} \om^{[n]}
\end{equation}
and 
\[
    c_{\ld,0}(l) := \EH_{\ld,0,0}(l) - \ld \log \BFV_0(l). 
\]
The first variation $(\dd \EH_{\ld,a,b})_l(\dot{l})$ converges locally uniformly in $l$ towards $(\dd \EH_{\ld,0,b})_l(\dot{l})$ as $a \to 0$.
\end{lem}
\begin{proof}
By Lemma~\ref{lem:Scal_lin}, 
\[
    L_{\ld,0,l}(\dot{l}) = \Scal_{e^l}(\om) \dot{l} + e^l \Dt_{\om, e^l}(\dot{l})
\]
and thus, 
\[
    (\dd \BFS_{0})_l(\dot{l}) 
    = \int_X \Scal_{e^l}(\om) \dot{l} \om^{[n]} + \int_X e^l \Dt_{\om, e^l}(\dot{l}) \om^{[n]} 
    = \int_X \Scal_{e^l}(\om) \dot{l} \om^{[n]}.
\]
We also have 
\[
    (\dd \BFV_0)_l(\dot{l}) 
    = \int_X e^l \dot{l} \om^{[n]}.
\]
Thus 
\begin{align*}
    &(\dd \EH_{\ld,0,b})_l(\dot{l}) \\
    &\quad= \frac{1}{\BFV_0(l)} \int_X \Scal_{e^l}(\om) \dot{l} \om^{[n]} 
    - \frac{\BFS_0(l)}{\BFV_0^2(l)} \int_X e^l \dot{l} \om^{[n]} \\
    &\qquad + \ld \left(-\frac{\int_X e^l \dot{l} + l e^l \dot{l} \om^{[n]}}{\BFV_0(l)} + \frac{\int_X l e^l \om^{[n]} \int_X e^l \dot{l} \om^{[n]}}{\BFV_0(l)^2} + \frac{\int_X e^l \dot{l} \om^{[n]}}{\BFV_0(l)}\right)\\
    &\qquad - \frac{b \int_X e^l \dot{l} \om^{[n]}}{\BFV_0(l)^2}\\
    &\quad= \frac{1}{\BFV_0(l)} \int_X \left(\Scal_{e^l}(\om) - \ld l e^l - e^l \left[\frac{\BFS_0(l)}{\BFV_0(l)} - \ld \frac{\int_X l e^l \om^{[n]}}{\BFV_0(l)} + \frac{b}{\BFV_0(l)}\right]\right) \dot{l} \om^{[n]}. 
\end{align*}
Set 
\[
    c_{\ld,0} := \frac{\BFS_0(l)}{\BFV_0(l)} - \ld \frac{\int_X l e^l \om^{[n]}}{\BFV_0(l)} 
    = \EH_{\ld,0,0}(l) - \ld \log \BFV_0(l).
\]

\smallskip
We now show the locally uniform convergence of $(\dd \EH_{\ld,a})_l(\dot{l})$ to $(\dd \EH_{\ld,0})_l(\dot{l})$.
From \eqref{eq:Futa_defn} and \eqref{eq:Fut0_defn}, it suffices to show that the following convergence locally uniformly
\[
    \frac{F_1(a,l) - F_2(a,l)}{a} \xrightarrow[a \to 0]{} l e^l
    \quad\text{and}\quad
    c_{\ld,a}(l) \xrightarrow[a \to 0]{} c_{\ld,0}(l).
\]
The first convergence above follows directly from \eqref{eq:Fk_exp}. 
We then verify the second convergence.
Recall that $c_{\ld,a}(l) = \BFV_a(l)^a \EH_{\ld,a}(l) - \frac{\ld}{a}(\BFV_a(l)^a - 1)$. 
Obviously, the first part converges locally uniformly towards $\EH_{\ld,0}(l)$, we only need to handle the second term. 
By \eqref{eq:Va_exp}, we have
\[
    \BFV_a(l) = \BFV_0(l) 
    + a \left(\BFV_0(l) \log \BFV_0(l) + \int_X e^l(-l^2/2 - l) \om^{[n]}\right)
    + O(a^2).
\]
Hence, 
\begingroup\allowdisplaybreaks
\begin{align*}
    \BFV_a(l)^a 
    &= \exp\left(a \log \left(\BFV_0(l) 
    + a \left(\BFV_0(l) \log \BFV_0(l) + \int_X e^l(-l^2/2 - l) \om^{[n]}\right)
    + O(a^2)\right)\right)\\
    &= \exp \left(a \log \BFV_0(l) + a \log \left(1 + a \left[\log \BFV_0(l) + \frac{\int_X e^l(-l^2/2 - l) \om^{[n]}}{\BFV_0(l)}\right] + O(a^2)\right) \right)\\
    &= \exp\left(a \log \BFV_0(l) +O(a^2)\right) \\
    &= 1 + a \log \BFV_0(l) + O(a^2). 
\end{align*}
\endgroup
From the above expansion, one can easily see the convergence of $c_{\ld,a}(l) \to c_{\ld,0}(l)$.
\end{proof}

By arguments similar to those in Lemma~\ref{lem:EH_C0conv} and Lemma~\ref{lem:EH0_1stvar}, one verifies that $\EH_{\ld,a,b}$ converges to $\EH_{\ld,0,0}$ in $\CC^2_{\loc}$. 
The computations follow the same ideas as in the proofs of the $\CC_{\loc}^0$ and $\CC_{\loc}^1$ convergence, and we therefore omit the details here. 
Combining these results, we obtain the following proposition. 

\begin{prop}\label{prop:EH_C2loc}
The Einstein--Hilbert functional $\EH_{\ld,a,b}(l)$ converges in $\CC^2_{\loc}$ to $\EH_{\ld,0,0}(l)$ when $(a,b) \to (0,0)$, and 
\[
    (\dd \EH_{\ld,a,b})_{l}(\dot{l}) = \frac{\Fut_{\ld,a,b,l}(\dot{l})}{\BFV_a(l)}.
\]
\end{prop}

\section{Construction of scalar-flat cones}\label{sec:4}
In this section, we use an implicit function theorem at infinity to construct scalar-flat K\"ahler cones starting from extremal K\"ahler metrics.
To invoke a LeBrun--Simanca type implicit function theorem (Lemma~\ref{l:(v,w)-LeBrun-Simanca}), it is essential to verify the vanishing of the weighted Futaki invariant.
The Einstein--Hilbert functional introduced in the previous section plays a crucial role in establishing the vanishing of the Futaki invariants for the relevant family of weights.

\smallskip
Let $(X, L)$ be a polarized orbifold of complex dimension $n$ with $\af = 2 \pi c_1(L)$. 
Suppose that $X$ admits an extremal orbifold K\"ahler metric in $\af$. 
We define the following rescaled functional of $\EH_{\ld,a,b}$: 
\begin{align*}
    \widetilde{\EH}_{\vep,a,b}(l) 
    &:= \vep^{-1} \left(\EH_{\vep^{-1},a,b}(\vep l) - \vep^{-1} \cdot \frac{1}{a} \left(1 - \frac{\BFV_0(0)}{\BFV_a(0)}\right) - \bar{s} - \frac{b}{\BFV_a(0)}\right) \\
    &= \frac{1}{\vep} \left(\frac{\BFS_a(\vep l)}{\BFV_{a}(\vep l)} + \frac{1}{\vep a} \left(\frac{\BFV_0(0)}{\BFV_a(0)} - \frac{\int_X F_0(a,\vep l) \om^{[n]}}{\BFV_a(\vep l)}\right) - \bar{s} + b\left(\frac{1}{\BFV_a(\vep l)} - \frac{1}{\BFV_a(0)}\right) \right). 
\end{align*}
In the following, we consider the rescaled Einstein--Hilbert functionals $\widetilde{\EH}_{\vep,a,b}$  restricted to the subspace of \emph{normalized} affine-linear functions $l$ by requiring $\int_{\Pol_{\af}} l \, \dd\mu_{\mathrm{DH}} = 0$ where $\mu_{\mathrm{DH}} = (\mu_\om)_\ast \om^{[n]}$ is the Duistermaat--Heckman measure. 
Using similar arguments as in Lemma~\ref{lem:EH_C0conv} and Lemma~\ref{lem:EH0_1stvar}, one can verify that $\widetilde{\EH}_{\vep,a,b}$ converges to $\widetilde{\EH}_{0,0,0}$ in $\CC^2_{\loc}$ when $(\vep,a,b) \to (0,0,0)$. 

\begin{lem}\label{l:conv}
The functional 
$\widetilde{\EH}_{\vep,a,b}$ converges to $\widetilde{\EH}_{0,0,0}$ in $\CC^2_{\loc}$ as $(\vep,a,b) \to (0,0,0)$. 
The limiting functional admits the explicit expression 
\begin{align*}
    \widetilde{\EH}_{0,0,0}(l)
    &= \frac{1}{\BFV_0(0)} \int_X \left(\Scal(\om) - \bar{s} - \frac{1}{2} (l - \overline{l})\right) l \om^{[n]}\\
    &= - \frac{1}{2\BFV_0(0)} \left(\int_X \left((\Scal(\om) - \bar{s}) - (l - \overline{l})\right)^2 \om^{[n]} - \int_X (\Scal(\om) - \bar{s})^2 \om^{[n]}\right). 
\end{align*}
where $\overline{l} = \BFV_0(0)^{-1} \int_X l \om^{[n]}$. 
\end{lem}

Note that the Hessian of $\widetilde{\EH}_{0,0,0}$, 
\[
    (\dd^2 \widetilde{\EH}_{0,0,0})_l(\dot{l}_1, \dot{l}_2) 
    = - \frac{1}{\BFV_0(0)} \int_X (\dot{l}_1 - \overline{\dot{l}_1}) (\dot{l}_2 - \overline{\dot{l}_2}) \om^{[n]}
\]
is strictly concave on the space of normalized affine-linear functions $\ell$. 
If $\om$ is an extremal K\"ahler metric in $\af$, we have  
\[ 
    \Scal(\om) = \mu_\om^* l_{\ext} 
    := 
    \langle\xi_{\ext}, \mu_\om\rangle +c_{\ext}, 
\]
and $\mathring{l}_{\ext} := \ell_{\rm ext} - \overline{\ell}_{\rm ext}$ is the unique normalized maximizer for $\widetilde{\EH}_{0,0,0}(\bullet)$.

\begin{rmk}
When $a=b=0$, $\widetilde{\EH}_{\vep,0,0}$ recovers the functional $W_\vep^\dagger$ in Inoue's work \cite[page~42]{Inoue_2022}.
This functional can be regarded as a scaled version of Perelman's $W$-functional (cf. \cite[Remark~3.8]{Inoue_2022}). 
\end{rmk}

\begin{thm}\label{thm:SFC}
Let $(X, L_X, \omega_X)$ be an $n$-dimensional polarized orbifold endowed with   an extremal orbifold K\"ahler metric  $\omega_X \in 2\pi c_1(L_X)$. 
There exist constants $\ld_+ > 0$, $\ld_- < 0$ and $N_0 \in \BN_{>0}$ such that, for all 
\[
    \ld \in (-\infty, \ld_-) \cup (\ld_+, +\infty) 
    \quad\text{and}\quad
    N \in (N_0, +\infty) \cap \BN,
\] 
if $(Z, L_Z, \om_Z)$ is a compact polarized K\"ahler manifold satisfying the following conditions: 
\begin{itemize}
    \item $Z$ is of complex dimension $(N-n)$, 
    \item $\om_Z\in 2\pi c_1(L_Z)$ is a K\"ahler metric of constant scalar curvature $\Scal(\om_Z) = \ld N$,
\end{itemize}
then the unitary bundle $\Sa_{\lambda, N}$ in $(\pi_X^* L_X \otimes \pi_Z^* L_Z)^{-1}\to X\times Z$ with respect to the hermitian metric with curvature $\omega_X \oplus \omega_Z$ admits a Sasaki structure of constant scalar curvature. Furthermore,  the constants $\lambda_{\pm}$ can be so chosen that if $\lambda > \ld_+$ the transversal scalar curvature of $\Sa_{\lambda, N}$ is positive,  whereas in the case $\ld < \ld_-$ the transversal scalar curvature of $\Sa_{\lambda, N}$ is negative.  In particular, if  $\ld > \ld_+$, the cone $\overline{((\pi_X^*L_X \otimes \pi_Z^* L_Z)^{-1})^\times}$ associated to the polarized orbifold $(X \times Z, \pi_X^*L_X \otimes \pi_Z^*L_Z)$ admits a scalar-flat K\"ahler cone metric. 
\end{thm}

\begin{proof} 
We work on $(X, L_X, \T)$ and drop the subscript $X$ to ease notation. 
Thus $L=L_X$ is the polarization of $X$ and, without loss of generality,  $\omega = \omega_X \in 2\pi c_1(L)$ is a $\T$-invariant extremal K\"ahler metric, viewed as $(1, \ell_{\rm ext})$-cscK metric (see Lemma~\ref{l:extremal}), where $\ell_{\rm ext}(x) = \langle\xi_{\ext},x\rangle + c_{\ext}$ denotes the Mabuchi--Futaki affine linear function on $\Pol_{L}$. 

\smallskip
We consider the normalized Einstein--Hilbert functionals $\widetilde{\bf EH}_{\varepsilon, a, b}(\bullet)$ defined on the open space of normalized affine-linear functions ${\rm Aff}_0(\Pol_L)$ satisfying 
\[ 
    0< 1 + a \ell(x) <2 \quad \text{on $\Pol_L$}, \qquad \int_{\Pol_L} l (x) \, \dd\mu_{\rm DH}=0,
\]
where 
$\mu_{\rm DH} = (\mu_\om)_\ast \om^{[n]}$ is the induced Duistermaat--Heckman measure on $\Pol_L$.
Recall that $\widetilde{\EH}_{0,0,0}(\bullet)$ is strictly concave on ${\rm Aff}_0(\Pol_L)$ and $\mathring{l}_{\ext}$ is its unique global maximizer. 
Since $\widetilde{\EH}_{\vep,a,b}(\bullet)$ converges to $\widetilde{\EH}_{0,0,0}(\bullet)$ locally uniformly in $\CC^2$ as $(\vep, a, b) \to (0,0,0)$, by the implicit function theorem, there exist an open neighborhood $U$ of $(0,0,0)$ in $\BR^3$ and a differentiable function $U \ni (\vep,a,b) \mapsto l_{\vep,a,b} \in {\rm Aff}_0(\tor^*)$, such that $l_{0,0,0} =  \mathring{l}_{\ext}$ and $(\dd \widetilde{\EH}_{\vep,a,b})_{l_{\vep,a,b}}(\bullet) = 0$ for all $(\vep,a,b) \in U$. 
This shows that 
\begin{equation}\label{eq:Fut=0mainthm}
    \Fut_{\vep^{-1},a,b,\vep l_{\vep,a,b}}(\xi) = 0, 
    \quad \forall \xi \in {\rm Aff}_0(\tor^*).
\end{equation}
 
\smallskip
In order to apply Lemma~\ref{l:(v,w)-LeBrun-Simanca}, we need to make sure that  the Futaki invariant $\Fut_{\vep^{-1},a,b,\vep l_{\vep,a,b}}(\xi)$ vanishes  for all affine-linear functions.  To achieve this, we shall choose the constant $b = b(\vep,a)$ such that 
\begin{equation}\label{eq:Fut_adjust}
    \Fut_{\vep^{-1}, a, b(\vep,a), \vep l_{\vep,a,b(\vep,a)}}(1) = 0.
\end{equation}  
The latter is equivalent to $\mathfrak{F}(\vep,a,b(\vep,a)) = 0$, where
\[
    \mathfrak{F}(\vep,a,b) 
    :=
    c_{\vep^{-1},a}(\vep l_{\vep,a,b}) + \frac{b}{\int_X F_1(a,\vep l_{\vep,a,b}) \om^{[n]}} 
    - d_{\vep^{-1},a}(\vep l_{\vep,a,b}),
\]
and
\begin{equation}\label{eq:d_defn}
    d_{\ld,a}(l) := \frac{1}{\int_X F_2(a,l) \om^{[n]}} \int_X \left(\Scal_{F_1(a,l)}(\om) - \frac{\ld}{a} (F_1(a,l)-F_2(a,l))\right) \om^{[n]}.
\end{equation}
To produce such a subfamily $(\vep, a, b(\vep,a))$ passing through $(0,0,0)$ such that $\mathfrak{F}(\vep, a, b(\vep,a)) = 0$ via the implicit function theorem, it suffices to verify that $\mathfrak{F}(0,0,0) = 0$ and $\pl_b \mathfrak{F}(0,0,0) \neq 0$. 
Observe that 
\begin{align*}
    c_{\vep^{-1},a}(\vep l) - d_{\vep^{-1},a}(\vep l) 
    &= \frac{\int_X \Scal_{F_1(a,\vep l)}(\om) a \vep l \om^{[n]}}{\int_X F_1(a,\vep l)\om^{[n]}}\\
    &\quad+ \frac{\left(\int_X (F_2(a,\vep l) - F_1(a,\vep l)) \om^{[n]}\right) \int_X \Scal_{F_1(a,\vep l)}(\om) \om^{[n]}}{\int_X F_1(a,\vep l) \om^{[n]} \int_X F_2(a,\vep l) \om^{[n]}} \\
    &\quad - \frac{1}{\vep a} \left(\frac{\int_X F_0(a,\vep l) \om^{[n]} \int_X F_2(a,\vep l) \om^{[n]} - (\int_X F_1(a,\vep l)\om^{[n]})^2}{\int_X F_1(a,\vep l) \om^{[n]} \int_X F_2(a,\vep l) \om^{[n]}}\right)
\end{align*}
which converges to $0$ as $(\vep,a) \to (0,0)$. 
Therefore, we have $\mathfrak{F}(0,0,b) = \frac{b}{\BFV_0(0)}$, and thus, $\mathfrak{F}(0,0,0) = 0$ and  $\pl_b \mathfrak{F}(0,0,0) = \frac{1}{\BFV_0(0)} > 0$. 
Putting now $l_{\vep,a} := l_{\vep,a,b(\vep,a)}$, we then obtain, by \eqref{eq:Fut=0mainthm} and \eqref{eq:Fut_adjust}, 
\begin{equation}\label{eq:Fut=0True}
    \Fut_{\vep^{-1}, a, b(\vep,a), \vep l_{\vep,a}}(\bullet) \equiv 0
\end{equation}
for all affine linear functions.

\smallskip
Combining \eqref{eq:Fut=0True} and  Lemma~\ref{l:(v,w)-LeBrun-Simanca}, we obtain a differentiable $(\vep,a)$-family of $\T$-invariant K\"ahler metrics $(\om_{\vep,a})_{(\vep,a) \in U}$ satisfying 
\[
    \Scal_{(1 + a \vep l_{\vep,a})^{\frac{1-a}{a}}}(\om_{\vep,a}) 
    = \mu_{\omega_{\varepsilon, a}}^*\left(\frac{1}{\vep a} ((1 + a\vep l_{\vep,a})^{\frac{1-a}{a}} - (1+a \vep l_{\vep,a})^{\frac{1-2a}{a}}) + d_{\vep^{-1},a}(\vep l_{\vep,a}) (1+a \vep l_{\vep,a})^{\frac{1-2a}{a}}\right).
\]

\smallskip
Consider an open interval $I \subset \BR$ containing $0$ and $N_0 \in \BN_{>0}$ such that $(\vep, -1/N) \in U$ for all $\vep \in I$ and $N \in (N_0, +\infty) \cap \BN$. 
Denote by $\ld = 1/\vep$ for $\vep \in I \setminus \{0\}$ and set 
\[
    \om_{\ld,N} := \om_{\vep,-1/N}
    \quad\text{and}\quad
    l_{\ld,N} := \frac{l_{\vep,-1/N}}{\ld}.
\] 
Then the metric $\om_{\ld,N}$ can be regarded as a $(v_{\ld,N}, \widetilde{w}_{\ld,N})$-cscK metrics for 
\begin{align*}
    v_{\ld,N} &:= \left(1- \frac{l_{\ld,N}}{N}\right)^{-(N+1)}
\end{align*}
and
\begin{align*}
    \widetilde{w}_{\ld,N} 
    &:= -\ld N \left(\left(1 - \frac{l_{\ld,N}}{N}\right)^{-(N+1)} - \left(1 - \frac{l_{\ld,N}}{N}\right)^{-(N+2)}\right) + d_{\ld,-1/N}(l_{\ld,N}) \left(1 - \frac{l_{\ld,N}}{N}\right)^{-(N+2)}\\
    &= \left(1 - \frac{l_{\ld,N}}{N}\right)^{-(N+2)} 
    \left(d_{\ld,-1/N}(l_{\ld,N}) + \ld l_{\ld,N}\right).
\end{align*}

\smallskip
Let $Z$ be an $(N-n)$-dimensional compact K\"ahler manifold and let $\om_Z$ be a cscK metric on $Z$ such that $\Scal(\om_Z) = \ld N$. 
The product metric of $\om_{\ld,N}$ and $\om_Z$ is then a $(v_{\ld,N}, w_{\ld,N})$-cscK metric on $X \times Z$, for 
\begin{equation}\label{eq:toSFCweight}
    w_{\ld,N} 
    = \widetilde{w}_{\ld,N} + v_{\ld,N}  \Scal(\om_Z) 
    = \left(1- \frac{l_{\ld,N}}{N}\right)^{-(N+2)} \left(d_{\ld,-1/N}(l_{\ld,N}) + \ld N\right).
\end{equation}

\smallskip
It remains to verify that $(d_{\ld,-1/N}(l_{\ld,N}) + \ld N)$ in \eqref{eq:toSFCweight} is positive (resp. negative) for all sufficiently large $N$ and for $\ld$ sufficiently positive (resp. negative). 
The arguments for the positive and negative cases are analogous; therefore, we present only the proof of the positive case below. 
By \eqref{eq:d_defn}, we get
\[
    d_{\ld,-1/N}(l_{\ld,N}) + \ld N
    = \frac{\int_X \Scal_{F_1(-\frac{1}{N}, l_{\ld,N})}(\om) \om^{[n]}}{\int_X F_2(-\frac{1}{N}, l_{\ld,N}) \om^{[n]}} 
    + \ld N \frac{\int_X F_1(-\frac{1}{N}, l_{\ld,N}) \om^{[n]}}{\int_X F_2(-\frac{1}{N}, l_{\ld,N}) \om^{[n]}}.
\]
As $l_{\ld,N} \to 0$ when $\ld, N \to +\infty$, we have the following limiting properties 
\[
    \frac{\int_X \Scal_{F_1(-\frac{1}{N}, l_{\ld,N})}(\om) \om^{[n]}}{\int_X F_2(-\frac{1}{N}, l_{\ld,N}) \om^{[n]}} \xrightarrow[\ld,N \to +\infty]{} \frac{\int_X \Scal(\om) \om^{[n]}}{\BFV_0(0)} = \bar{s}
\]
and
\[
    \frac{\int_X F_1(-\frac{1}{N}, l_{\ld,N}) \om^{[n]}}{\int_X F_2(-\frac{1}{N}, l_{\ld,N}) \om^{[n]}}
    \xrightarrow[\ld, N \to +\infty]{} 1.
\]
Therefore, as $\ld, N$ are sufficiently large, the quantity $(d_{\ld,-1/N}(l_{\ld,N}) + \ld N)$ is positive. 
By Corollary~\ref{c:tool}, this induces a Sasaki structure of positive constant transversal scalar curvature. 
\end{proof}

Combining all the material mentioned before, we easily derive the proofs for Theorems \ref{thm:main} and \ref{thm:Sasaki} in the introduction. 

\begin{proof}[Proof of Theorem~\ref{thm:main}]
Consider an integer $N > N_0$ sufficiently large such that 
\[
    \ld_N := \frac{2(N-n)(N-n+1)}{N} > \ld_+.
\] 
Take $k := N-n$ and $Z := \BP^k$. 
Let $\om_Z$ be the Fubini--Study metric in $2\pi c_1(\CO_{\BP^{k}}(1))$. 
We have 
\[
    \Scal(\om_Z) = 2k(k+1) = 2(N-n)(N-n+1) = \ld_N  N.
\]
This fulfills the condition in Theorem~\ref{thm:SFC} and hence  the cone 
\[
    ((\pi_X^* L \otimes \pi_{\BP^k}^* \CO_{\BP^{k}}(1))^{-1})^\times
\] 
over $X \times \BP^k$ admits a scalar-flat K\"ahler cone metric. 
\end{proof}

\begin{proof}[Proof of Theorem~\ref{thm:Sasaki}]
Let $(\Sa, \hat\chi, \Ds, J)$ be a quasi-regular compact $(2n+1)$-dimensional Sasaki manifold such that the transversal K\"ahler geometry is Calabi-extremal.  In this case, 
$X := \Sa/\Sph^1_{\hat\chi}$ is a polarized 
K\"ahler orbifold endowed with an extremal K\"ahler metric $\om_X \in 2\pi c_1(L_X)$ induced by the transversal K\"ahler structure of $\Sa$. 
The Fubini--Study metric of scalar curvature $2k(k+1)$ on $\BP^k$ naturally gives raise to the flat K\"ahler cone $\CO_{\BP^k}(-1)^\times = \C^{k+1}\setminus \{0\}$ and the regular Sasaki sphere $(\Sph^{2k+1}, \hat \chi_k, \Ds_k, J_k)$. 
Following the Sasaki join construction discussed in Section~\ref{subsec:Sasaki_joint}, $\Sa \star_{1,1} \BS^{2k+1}$ (with the Sasaki--Reeb vector field $\hat{\xi}_k = \frac{1}{2}(\hat\chi + \hat\chi_k)$) is associated to the cone 
\[
    \CC_k = \left((\pi_X^* L_X \otimes \pi_{\BP^k}^* \CO_{\BP^k}(1))^{-1}\right)^{\times}.
\] 
By Theorem~\ref{thm:main}, there exists $k_0 > 0$ such that $\CC_k$ admits a scalar-flat K\"ahler cone metric for all $k > k_0$. 
Combining this with Corollary~\ref{c:tool}, one can derive that $\Sa \star_{1,1} \BS^{2k+1}$  admits a Sasaki--Reeb vector field $\hat\xi$ and a compatible Sasaki structure of constant positive transversal scalar curvature. 
\end{proof}

Theorem~\ref{thm:SFC} can also produce Sasaki structures with negative constant positive transversal scalar curvature. 

\begin{cor}\label{c:negative cscS}
Let $(\Sa, \hat\chi, \Ds, J)$ be a $(2n+1)$-dimensional smooth compact quasi-regular Sasaki manifold with Calabi-extremal transversal K\"ahler structure, and let $(X, L_X)$ be the induced  $n$-dimensional polarized orbifold $\Sa/\Sph^1_{\hat \chi}$ endowed with an extremal K\"ahler metric $\omega_X \in 2 \pi c_1(L_X)$ as in the hypotheses  of Theorem~\ref{thm:SFC}. 
Then, for any given integer $N \in (N_0, +\infty) \cap \BN_{>0}$, there exist $g_N \gg 0$ such that for any compact Riemann surface $\Sm_g$ of genus $g \geq g_N$ and constant scalar curvature $4(1-g)$, the regular $(2(N-n)+1)$-dimensional Sasaki manifold $\Sa_{N, g}$ obtained as the unitary bundle of $\left(\bigotimes_{j=1}^{N-n}\pi_j^* K_{\Sm_g}^{\frac{1}{2-2g}}\right) \to (\Sm_g)^{N-n}$ with respect to the hermitian metric with curvature the $(N-n)$-th product K\"ahler metric on $(\Sm_g)^{N-n}$,   satisfies  that the Sasaki join
\[ \tilde \Sa =\Sa \star_{1,1} \Sa_{N, g} \]
admits a Sasaki--Reeb vector field $\hat \xi$ and a  Sasaki structure of negative constants transversal scalar curvature in $\Xi^{\hat \T}_{\hat \xi, \eta_0^{\hat \xi}, J^{\hat \xi}}(\tilde \Sa)$.
\end{cor}

\begin{proof} 
Let $(X, L_X, \omega_X)$ be the polarized extremal K\"ahler orbifold quotient of $(\Sa, \hat \chi, \Ds, J)$ by $\Sph_{\hat \chi}^1$,  and $\Sm_g$ be a compact complex curve of genus $g > 1$ and  $\af$ the positive integral class generating $H^2(\Sm_g, \BZ)$.
We fix the K\"ahler metric $\om_g \in 2\pi \af$ with constant scalar curvature $4(1- g)$. 
Let $L_\af$ be a $(2g-2)$-root of $K_{\Sm_g}$ which we denote by $K_{\Sm_g}^{\frac{1}{2g-2}}$. 

\smallskip
Given an integer $N \in (N_0,+\infty)$, the product K\"ahler metric $\om_{Z} = \om_g \oplus \cdots \oplus \om_g$ on $Z := (\Sigma_g)^{N-n}$ has scalar curvature $4(N-n)(1-g)$. 
Then for $g \gg 1$, we have
\[
    \ld_{N,g} := \frac{(N-n)(4-4g)}{N} < \ld_-.
\]
The polarization $L_Z := \bigotimes_{i=1}^{N-n} \pi_i^\ast L_\af = \bigotimes_{i=1}^{N-n} \pi_i^\ast K_{\Sm_g}^{\frac{1}{2g-2}}$ of $Z$ with the constant scalar curvature K\"ahler metric $\omega_Z \in 2\pi c_1(L_Z)$ then fulfill the conditions in Theorem~\ref{thm:SFC}. 
Thus, the unitary bundle in $(\pi_X^*L_X  \otimes \pi_Z^*L_Z)^{-1} \to X\times Z$ admits a Sasaki structure of constant negative transversal scalar curvature. 
\end{proof}

\bibliographystyle{alpha}
\bibliography{biblio}
\end{document}